\documentclass[a4paper,11pt,DIV=12,
abstract=on
]{scrartcl}
\usepackage{mathtools}
\mathtoolsset{showonlyrefs}
\usepackage{amsmath, amsthm, amssymb}
\usepackage{fontenc}
\usepackage[utf8]{inputenc}
\usepackage{graphicx}
\usepackage{subcaption,tikz,pgfplots,ifsym}
\DeclareCaptionLabelSeparator{periodspace}{.\ }
\definecolor{rbcolor}{rgb}{0.063,0.278,0.4}
\usepackage[backgroundcolor=rbcolor!25!white, bordercolor=rbcolor!75, linecolor=rbcolor!75,textsize=small]{todonotes}
\usepackage{multirow}
\usepackage{algorithm, booktabs,multirow}
\usepackage[noend]{algpseudocode}

\usepackage{color}

\setkomafont{sectioning}{\rmfamily\bfseries}
\setkomafont{title}{\rmfamily}
\usepackage{array}
\usepackage{enumitem,listings}
\setlist{leftmargin=0pt,labelwidth=*,
,align=left%
,itemsep=0pt}
\usepackage[unicode]{hyperref}

\tikzstyle{point}=[inner sep=.2ex,fill=black,shape=circle] 
\tikzstyle{pointc}=[inner sep=.2ex,fill=blue,shape=circle] 
\tikzstyle{pointb}=[inner sep=.2ex,fill=blue!50!black,shape=circle] 
\tikzstyle{helping line}=[thin, black!50] 
\tikzstyle{diam}=[thin, black!50,dashed] 
\tikzstyle{ticklabel}=[right=-1pt,font=\tiny]
\tikzstyle{axis}=[very thin, font=\footnotesize]
\tikzstyle{ticks}=[very thin, font=\footnotesize]

\newlength{\scalepi}
\newtheorem{theorem}{Theorem}[section]

\newtheorem{proposition}[theorem]{Proposition}
\newtheorem{lemma}[theorem]{Lemma}
\newtheorem{remark}[theorem]{Remark}
\newtheorem{example}[theorem]{Example}
\newtheorem{corollary}[theorem]{Corollary}

\newcommand{\I}{\mathbb{I}}

\newcommand{\R}{\mathbb{R}}

\newcommand{\Z}{\mathbb{Z}}

\renewcommand{\atop}[2]{\genfrac{}{}{0pt}{}{#1}{#2}}

\DeclareFontFamily{U}{mathx}{\hyphenchar\font45}
\DeclareFontShape{U}{mathx}{m}{n}{
      <5> <6> <7> <8> <9> <10>
      <10.95> <12> <14.4> <17.28> <20.74> <24.88>
      mathx10
      }{}
\DeclareSymbolFont{mathx}{U}{mathx}{m}{n}
\DeclareMathSymbol{\bigtimes}{1}{mathx}{"91}
\newcommand{\tT}{\mathrm{T}}
\DeclareMathOperator*{\argmin}{arg\,min} 
\DeclareMathOperator{\prox}{prox} 

\DeclareMathOperator{\sgn}{sgn}

\newcommand{\abs}[1]{\left\lvert #1 \right\rvert}
\newcommand{\norm}[2]{\left\lVert #1 \right\rVert_{#2}}

\hypersetup{pdfauthor={Ronny Bergmann, Friederike Laus, Gabriele Steidl, and Andreas Weinmann},
 pdftitle={Second Order Differences of Cyclic Data and Applications in Variational Denoising},%
  pdfsubject={draft},%
  pdfcreator = {pdflatex and TextMate},%
  pdfkeywords={},
  plainpages=false, pdfstartview=FitH, pdfview=FitH, pdfpagemode=UseOutlines,%
  bookmarksnumbered=true,bookmarksopen=false,bookmarksopenlevel=0,%
  colorlinks=true,linkcolor=black,citecolor=black,urlcolor=black%
}
\title{Second Order Differences of Cyclic Data\\ and Applications in Variational Denoising}

\author{%
Ronny Bergmann\thanks{Fachbereich für Mathematik, Technische Universität %
Kaiserslautern, Paul-Ehrlich-Str. 31, 67663 Kaiserslautern, Germany, %
$\{$bergmann, steidl, friederike.laus$\}$@mathematik.uni-kl.de.}, %
Friederike Laus\footnotemark[1], %
Gabriele Steidl\footnotemark[1], %
Andreas Weinmann\thanks{%
Department of Mathematics, Technische Universität München
and
Fast Algorithms for Biomedical Imaging Group, Helmholtz-Zentrum %
München, Ingolstädter Landstr. 1, 85764 Neuherberg, Germany, %
andreas.weinmann@helmholtz-muenchen.de.}
}
 
\date{\today}

\begin{document}

\maketitle
\begin{abstract}
	\parindent0mm
	\parskip0mm
	\noindent
In many image and signal processing applications, as interferometric synthetic
aperture radar (SAR) or color image restoration in HSV or LCh spaces the data has
its range on the one-dimensional sphere $\mathbb S^1$.
Although the minimization of total variation (TV) regularized functionals is
among the most popular methods for edge-preserving image restoration such
methods were only very recently applied to cyclic structures.
However, as for Euclidean data, TV regularized variational methods
suffer from the so called staircasing effect. This effect can be avoided
by involving higher order derivatives into the functional.

This is the first paper which uses higher order differences of cyclic data
in regularization terms of energy functionals for image restoration.
We introduce absolute higher order differences for
$\mathbb S^1$-valued data in a sound way which
is independent of the chosen representation system on the circle.
Our absolute cyclic first order difference is just the
geodesic distance between points. Similar to the geodesic distances the
absolute cyclic second order differences have only values in $[0,\pi]$.
We update the cyclic variational TV approach by our new cyclic second order
differences. To minimize the corresponding functional we apply a cyclic
proximal point method which was recently successfully proposed for Hadamard
manifolds. Choosing appropriate cycles this algorithm can be implemented in an
efficient way. The main steps require the evaluation of proximal mappings of
our cyclic differences for which we provide analytical expressions.
Under certain conditions we prove the convergence of our algorithm.
Various numerical examples with artificial as well as real-world data demonstrate the
advantageous performance of our algorithm.
\end{abstract}
%
%
\section{Introduction} \label{sec:intro}
A frequently used method for edge-preserving image denoising is the variational
approach which minimizes the Rudin-Osher-Fatemi (ROF) functional~\cite{ROF92}.
In a discrete (penalized) form the ROF functional can be written as
\[
\sum_{i,j} ( f_{i,j} - x_{i,j} )^2
	+ \lambda \sum_{i,j}  \lvert\nabla x_{i,j}\rvert%
	, \quad \lambda > 0,
\]
where $f \in \mathbb R^{N,M}$ is the given corrupted image and $\nabla$ denotes
the discrete gradient operator which contains usually first order forward
differences in vertical and horizontal directions. The regularizing term
\(\sum_{i,j}  \lvert\nabla x_{i,j}\rvert\) can be considered as discrete version
of the total variation (TV) functional. Since the gradient does not penalize
constant areas the minimizer of the ROF functional tends to have such regions,
an effect known as staircasing. An approach to avoid this effect consists in
the employment of higher order differences/derivatives. Since the pioneering
work~\cite{CL97} which couples the TV term with higher order terms by infimal
convolution various techniques with higher order  differences/derivatives were
proposed in the literature, among
them~\cite{BKP09,CEP07,CMM00,DSS09,DWB09,HS06,LBU2012,LLT03,LT06,Sche98,SS08,SST11}.

In various applications in image processing and computer vision the functions
of interest take values on the circle $\mathbb S^1$ or another manifold.
Processing manifold-valued data has gained a lot of interest in recent years.
Examples are wavelet-type multiscale transforms for manifold data~
\cite{GW09,RDSDS05,Wein12} and manifold-valued partial differential equations~
\cite{CTDF04,GHS13}. Finally we like to mention statistical issues on
Riemannian manifolds~\cite{Fle13,FJ07,Pen06} and in particular the statistics
of circular data~\cite{fisher95,JS2001}.
The TV notation for functions with values on a manifold has been studied
in~\cite{GM06,GM07} using the theory of Cartesian currents. These papers were
an extension of the previous work~\cite{GMS93} were the authors focus
on $\mathbb S^1$-valued functions and show in particular the existence of
minimizers of certain energies in the space of functions with bounded total
cyclic variation. The first work which applies a cyclic TV approach among other
models for imaging tasks was recently published by Cremers and Strekalovskiy
in~\cite{SC11,CS13}. The authors unwrapped the function values to the real axis
and proposed an algorithmic solution to account for the periodicity. An
algorithm which solves TV regularized minimization problems on Riemannian
manifolds was proposed by Lellmann et al. in~\cite{LSKC13}. They reformulate
the problem as a multilabel optimization problem with an infinite number of
labels and approximate the resulting hard optimization problem using convex
relaxation techniques. The algorithm was applied for chromaticity-brightness
denoising, denoising of rotation data and processing of normal fields for
visualization. Another approach to TV minimization for manifold-valued data via
cyclic and parallel proximal point algorithms was proposed by one of the authors
and his colleagues in~\cite{WDS2013}. It does not require any labeling or
relaxation techniques.
The authors apply their algorithm in particular for diffusion tensor imaging
and interferometric SAR imaging. For Cartan-Hadamard manifolds convergence of
the algorithm was shown based on a recent result of 
Ba{\v{c}}{\'a}k~\cite{Bac13a}. Unfortunately, one of the
simplest manifolds that is not of Cartan-Hadamard type is the
circle~\(\mathbb S^1\).

In this paper we deal with the incorporation of higher order differences into
the energy functionals to improve denoising results for $\mathbb S^1$-valued
data.
Note that the (second-order) total generalized variation was
generalized for tensor fields in~\cite{VBK13}.
However, to the best of our knowledge this is the first paper which defines
second order differences of cyclic data and uses them in regularization terms
of energy functionals for image restoration.
We focus on a discrete setting. First we  provide a
meaningful definition of higher order differences for cyclic data which we call
\textit{absolute cyclic differences}.\enlargethispage{\baselineskip}
In particular our absolute cyclic first  order differences resemble the
geodesic distance (arc length distance) on the circle. As the geodesics the
absolute cyclic second order differences take only values in $[0,\pi]$.
This is not necessary the case for differences of order larger than two.
Following the idea in~\cite{WDS2013} we suggest a cyclic proximal point
algorithm to minimize the resulting functionals. This algorithm requires the
evaluation of certain proximal mappings. We provide analytical expression for
these mappings. Further, we suggest an appropriate choice of the cycles such
that the whole algorithm becomes very efficient. We apply our algorithm to
artificial data as well as to real-world interferometric SAR data.

The paper is organized as follows: in Section~\ref{sec:diff} we propose a
definition of differences on $\mathbb S^1$. Then, in Section~\ref{sec:prox}, we
provide analytical expressions for the proximal mappings required in our cyclic
proximal point algorithm. The approach is based on unwrapping the circle
to~\(\R\) and considering the corresponding proximal mappings
on the Euclidean space. The cyclic proximal point algorithm is presented in
Section~\ref{sec:cpp}. In particular we describe a vectorization strategy which
makes the Matlab implementation efficient and provides parallelizability,
and prove its convergence under certain assumptions.
Section~\ref{sec:numerics} demonstrates the
advantageous performance of our algorithm by numerical examples. Finally,
conclusions and directions of future work are given in
Section~\ref{sec:conclusions}.
%
\section{Differences of \texorpdfstring{$\mathbb S^1$}{𝕊¹}--valued data} \label{sec:diff}
%
Let $\mathbb S^1$ be the unit circle in the plane
\[
\mathbb S^1 := \{p_1^2 + p_2^2 = 1: p = (p_1,p_2)^\tT \in \R^2\}
\]
endowed with the {\itshape geodesic distance} (arc length distance)
\[
d_{\mathbb S^1} (p,q) = \arccos( \langle p,q \rangle ).
\]
Given a base point $q \in \mathbb S^1$, the {\itshape exponential map}~
$\exp_q: \R \rightarrow \mathbb S^1$ from the tangent
space~$T_q\mathbb S^1 \simeq \R$ of $\mathbb S^1$ at $q$ onto $\mathbb S^1$ is
defined by
\[
\exp_q(x) = R_x q, \qquad R_x :=
\begin{pmatrix}
\cos x & -\sin x\\
\sin x & \cos x
\end{pmatrix}.
\]
This map is $2\pi$-periodic, i.e.,
$\exp_q(x) = \exp_q((x)_{2\pi})$ for any $x \in \R$,
where
$(x)_{2\pi}$ denotes the unique point in $[-\pi,\pi)$ such that
$x = 2\pi k + (x)_{2\pi}$, $k \in \mathbb Z$.
Some useful properties of the mapping~
$(\cdot)_{2\pi}: \R \rightarrow [-\pi,\pi)$ (which can also be considered as
mapping from $\R$ onto $\R / 2\pi \Z$) are collected in the following remark.
%
\begin{remark} \label{lem:mod_pi}
The following relations hold true:
\begin{enumerate}
\item
\(\big( (x)_{2\pi} \pm (y)_{2\pi} \big)_{2\pi} = (x \pm y)_{2\pi} \qquad \)
for all \(x, y \in \R\).
\item
If
$
z = (x-y)_{2\pi}
$ then $
x = (z+y)_{2\pi}
\qquad
$ for all $
x \in [-\pi,\pi)
,\
y \in \R
$.
\end{enumerate}
While i) follows by straightforward computation relation ii) can be seen as
follows:
For $z = (x-y)_{2\pi}$ there exists $k\in \Z$ such that
		\begin{equation*}
		 x-y = (x-y)_{2\pi} +2\pi k = z + 2\pi k.
		\end{equation*}
		Hence it follows $x = z+y +2\pi k$ and since $x\in [-\pi,\pi)$ further
		\begin{equation*}
		 x = (x)_{2\pi} = (z+y +2\pi k)_{2\pi} = (z+y)_{2\pi}.\qedhere
		\end{equation*}
\end{remark}
To guarantee the injectivity of the exponential map, we restrict its domain of
definition from $\R$ to $[-\pi,\pi)$. Thus, for~$p,q \in \mathbb S^1$, there is
now a unique $x \in [-\pi,\pi)$ satisfying~$\exp_q(x) = p$. In particular we
have $\exp_q(0) = q$. Given such representation system~$x_j \in [-\pi,\pi)$ of
$p_j \in \mathbb S^1$, $ j = 1,2$ centered at an arbitrary point~$q$
on~$\mathbb S^1$ the geodesic distance becomes
\begin{equation} \label{d}
	d_{\mathbb S^1}(p_1,p_2) = d(x_1,x_2)
	= \min_{k \in \Z }\lvert x_2 - x_1 + 2\pi k\rvert
	= \lvert(x_2 - x_1)_{2\pi}\rvert.
\end{equation}
Actually we need only $k \in \{0,\pm 1\}$ in the minimum.
Clearly, this definition does not depend on the chosen center point $q$.

We want to determine general finite differences of $\mathbb S^1$-valued data.
Let $w = (w_j)_{j=1}^d \in \R^d \backslash \{ 0 \}$	with
\begin{equation} \label{moments}
\langle w ,1_d \rangle = \sum_{j=1}^d w_j = 0,
\end{equation}
where $1_d$ denotes the vector with $n$ components one.
We define the {\itshape finite difference operator} $\Delta(\cdot;w): \R^d \rightarrow \mathbb R$
by
\[
\Delta(x;w) := \langle x,w \rangle  \quad \mbox{for all} \; x \in \R^d.
\]
By~\eqref{moments}, we see that $\Delta(\cdot;w)$ vanishes for constant vectors
and is therefore translation invariant, i.e.,
\begin{equation} \label{ti}
\Delta(x + \alpha 1_d;w) = \Delta(x;w) \quad \mbox{for all} \; \alpha \in \R.
\end{equation}
\begin{example}\label{diff_n}
For the binomial coefficients with alternating signs
\[w = b_n := \left( (-1)^{j+n-1} { n\choose {j-1}} \right)_{j=1}^{n+1}\]
we obtain
the (forward) differences of order $n$:
\[
	\Delta(x;w)
	= \Delta_n(x) = \langle x,b_n \rangle
	= \sum_{j=1}^{n+1} (-1)^{j+n-1} { n\choose {j-1}} x_j.
\]
Note that $\Delta_n$ does not only fulfill~\eqref{moments}, but vanishes
exactly for all `discrete polynomials of order $n-1$', i.e., for all vectors
from $\operatorname{span}\{ (j^r)_{j=0}^n:r=0,\ldots,n-1\}$. Here we are
interested in first and second order differences
\begin{align*}
\Delta_1 (x_1,x_2) &= \Delta(x;b_1) = x_2 - x_1,\\
\Delta_2 (x_1,x_2,x_3) &=  \Delta(x;b_2) = x_1 - 2 x_2 + x_3.
\end{align*}
Moreover, we will apply the `mixed second order' difference with
$w = b_{1,1} := (-1,1,1,-1)^\tT$ and use the notation
\[
\Delta_{1,1} (x_1,x_2,x_3,x_4) = \Delta(x;b_{1,1})  = -x_1 + x_2 + x_3 - x_4.
\]
\end{example}
%
We want to define differences for points
$(p_j)_{j=1}^d \in (\mathbb S^1)^d$ using their representation
$x := (x_j)_{j=1}^d \in [-\pi,\pi)^d$  with respect to an arbitrary fixed
center point. As the geodesic distance~\eqref{d} these differences should be
independent of the choice of the center point.  This can be achieved if and
only if the differences are shift invariant modulo $2\pi$.  Let
$\I_d := \{1,\ldots,d\}$.
We define the {\itshape absolute cyclic difference}
of $x \in [-\pi,\pi)^d$ (resp. $(p_j)_{j=1}^d\in (\mathbb S^1)^d$)
with respect to $w$ by
\begin{equation} \label{def_cyc_diff}
	d(x;w) 
	:= \min_{\alpha \in \R}
		\bigl\lvert \Delta \big( [x + \alpha 1_d]_{2 \pi}; w  \big) \bigr\rvert
	= \min_{j \in \I_d}
		 \bigl\lvert \Delta \big( [x - (x_j+\pi) 1_d]_{2 \pi}  ; w  \big) \bigr\rvert,
\end{equation}
where $[x]_{2\pi}$ denotes the component-by-component application of
$(t)_{2\pi}$ if $t \not = (2k +1) \pi $, $k \in \Z$ and 
$[(2k +1) \pi]_{2\pi} = \pm \pi$, $k \in \Z$. The definition allows that points
having the same value are treated separately,
cf.~Figure~\ref{fig:choosing_sides}. This ensures that
$d(\cdot;w): (\mathbb S^1)^d \rightarrow \R$ is a continuous map. For example we
have $d\bigl( (-\pi,0,-\pi)^\tT;b_2 \bigr) = 0$.
Figures~\ref{fig:def-d}
and~\ref{fig:choosing_sides} illustrate definition~\eqref{def_cyc_diff}.
For the absolute cyclic differences related to the
differences in Example~\ref{diff_n} we will use the simpler notation
\begin{equation} \label{d2}
d_n(x) := d(x; b_n) \quad \text{and} \quad d_{1,1}(x) := d(x; b_{1,1}).
\end{equation}
\begin{figure}[tbp]\centering
	\begin{subfigure}{.8\textwidth}\centering
		\includegraphics{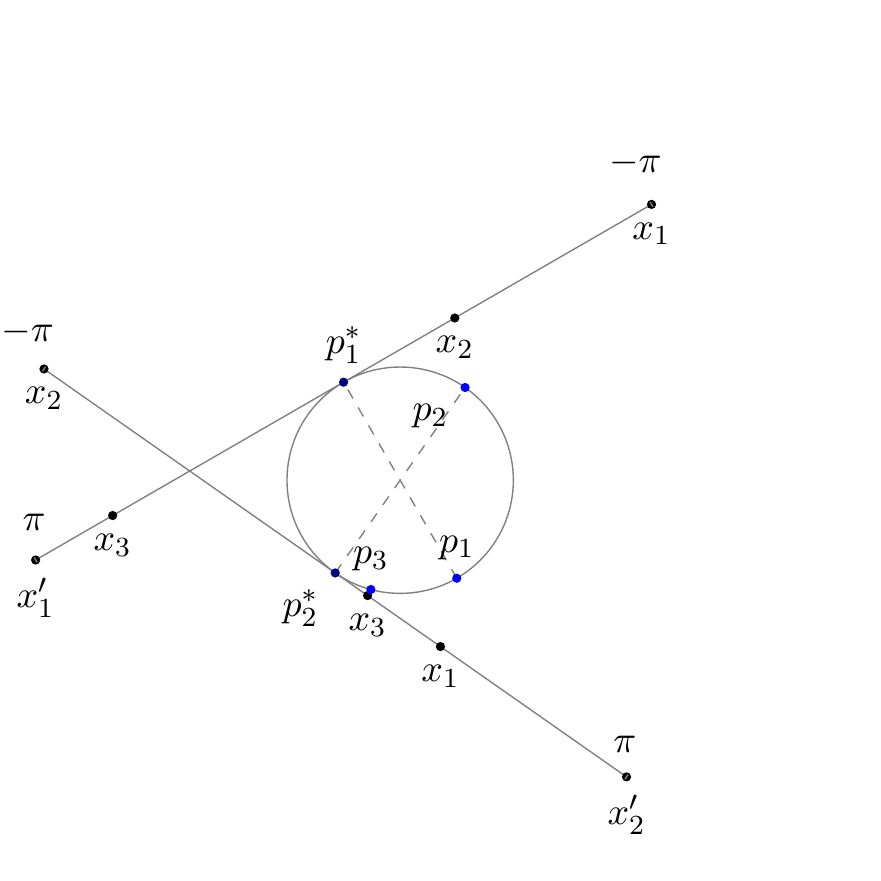}
		\vspace{-1.1\baselineskip}
		\caption{\({\exp}_{p_1^*}\) (top) and \({\exp}_{p_2^*}\) (bottom).}
		\label{subfig:p1}
	\end{subfigure}
	\begin{subfigure}{.8\textwidth}\centering
		\vspace{-\baselineskip}
		\includegraphics{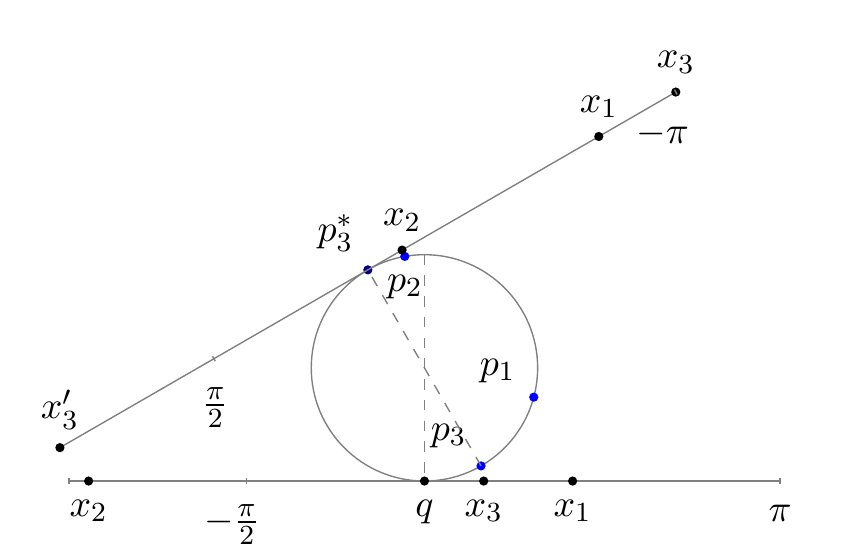}
		\vspace{-.2\baselineskip}
		\caption{\({\exp}_{p_3^*}\) (top) and \({\exp}_{q}\) (bottom).}
		\label{subfig:p3}
	\end{subfigure}
	\begin{subfigure}{.8\textwidth}\centering
		\includegraphics{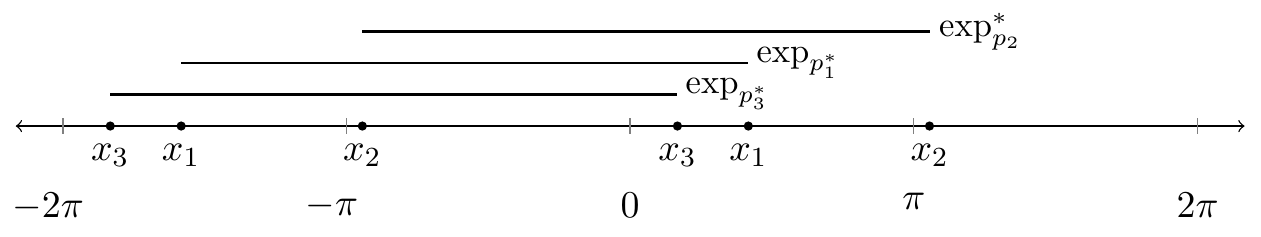}
		\vspace{-1.3\baselineskip}
		\caption{Settings from the tangential maps of \(p_j^*\), \(j=1,2,3\), on \(\mathbb R\) using the representation system according to~\(\exp_q\).}
		\label{subfig:forpartsonR}
	\end{subfigure}
	\vspace{-.5\baselineskip}
	\caption{Three points \(p_j\), \(j=1,2,3\), on the circle (blue) and their
	inverse exponential maps at \(p_j^*\), \(j=1,2,3\), (dark blue),
	where~$p_j^*$ denotes the antipodal point of $p_j$. In other words, we cut
	the circle at the point \(p_j\) and unwind it with respect to the tangent
	line at the antipodal point \(p_j^*\). The absolute cyclic differences take
	the three pairwise different positions of the points $x_j$, \(j=1,2,3\) to
	each other into account. These are shown in\subref{subfig:forpartsonR} with
	respect to the representation system from the arbitrary point~\(q\)
	in\subref{subfig:p3}.}
	\label{fig:def-d}
\end{figure}

\begin{figure}[tbp]\centering
		\includegraphics{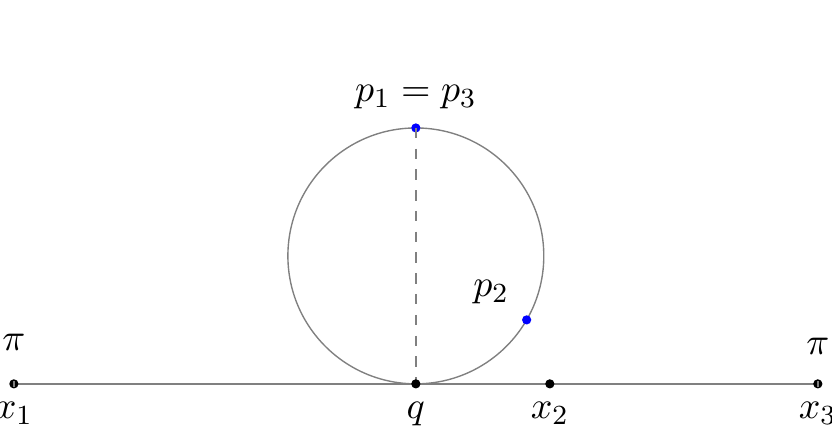}
	\caption{Three points  \(p_j\), $j=1,2,3$ on the circle, where \(p_1=p_3\)
	and ${\exp}_{q}$, \(q=p_1^*\).
	Though \(p_1,p_3\) denote the same point on the circle they are treated
	separately in the definition of the absolute cyclic differences.}
	\label{fig:choosing_sides}
\end{figure}

The following equivalent definition of absolute cyclic differences appears to
be useful.
\begin{lemma} \label{cyc_diff_perm}
Let $x \in [-\pi,\pi)^d$ be sorted in ascending order as
$-\pi \le x_{j_1} \le \ldots \le x_{j_d} < \pi$
and set
$x^1 := (x_{j_i})_{i=1}^d$.
Let $P$ denote the corresponding permutation matrix,
i.e.,
$
P x = x^1
$
and
$
x = P^\tT x^1.
$
Consider the $2\pi$ shifted versions of $x^1$ given by
\[
x^k = x^1 + 2\pi \sum_{j=1}^{k-1} e_j
\quad
k = 2,\ldots,d,
\]
where $e_j \in \R^d$ denotes the $j$-th unit vector.
Then it holds
\begin{equation}\label{dw}
	d(x;w)
	= \min_{k \in \I_d} \bigl\lvert\Delta(P^\tT x^k;w)\bigr\rvert
	= \min_{k \in \I_n} \Bigl\lvert\Delta(x^k;w)
		+ 2\pi\Bigl\langle \sum_{j=1}^{k-1} e_j, P w \Bigr\rangle\Bigr\rvert.
\end{equation}
\end{lemma}
\begin{proof}
The first equality in~\eqref{dw} follows directly by
definition~\eqref{def_cyc_diff}. To see the second one, note that by linearity
of the inner product we have
\begin{equation}\label{helpful}
	\langle P^\tT x^k,w \rangle
	= \langle P^\tT x^1, w \rangle
		+ 2\pi\Bigl\langle\sum_{j=1}^{k-1} e_j, P w \Bigr\rangle
	= \langle x, w \rangle + 2\pi\Bigl\langle \sum_{j=1}^{k-1} e_j, P w \Bigr\rangle.
\end{equation}
\end{proof}
For the geodesic distance we obtain by~\eqref{d} that~$d_1(x)
= \big\lvert \big( \Delta_1(x))_{2\pi} \big) \big\rvert$. In general the relation
\begin{equation} \label{important_rel}
	d(x;w)
	= \lvert(\langle x,w \rangle)_{2\pi}\rvert
	\qquad\text{ for all } x \in [-\pi,\pi)^d
	\end{equation}
does not hold true as the following example shows.
%
\begin{example} \label{conterex}
In general the $n$-th order absolute cyclic difference cannot be written as
$
d_n(x) = \lvert(\langle x,b_n \rangle)_{2\pi}\rvert
	= \lvert(\Delta_n(x))_{2\pi})\rvert.
$
Consider for example the absolute cyclic third order difference
for $x := \frac{\pi}{16}(-15,- 13, 12, 14)^\tT$
given by~\eqref{dw} as
\[
	d_3(x_1,x_2,x_3,x_4)
	= \min_{k=1,2,3,4} \Delta_3(x^k),
	\quad \Delta_3(x) = - x_1 + 3x_2 - 3x_3 + x_4.
\]
We obtain
\[
\Delta_3(x^1) = \Delta_3(x) = \frac{-46 \pi}{16}, \;
\Delta_3(x^2) = \Delta_3(x^4) = \frac{-78 \pi}{16} ,\;
\Delta_3(x^3) = \frac{18 \pi}{16},
\]
so that $d_3(x) = \frac{18 \pi}{16} > \pi$.
\end{example}
For $w \in \{b_2,b_{1,1} \}$  relation ~\eqref{important_rel} holds true by the
next lemma.
\begin{proposition} \label{diff_sec}
For $w \in \{b_2,b_{1,1} \}$ the following relation holds true:
\begin{equation} \label{sec_diff}
	d(x;w) 
	= \min_{k \in \Z} \lvert\Delta(x;w) + 2\pi k\rvert
	= \bigl\lvert \bigl( \Delta(x;w) \bigr)_{2\pi}\bigr\rvert.
\end{equation}
\end{proposition}
Note that we need only the minimum over $k \in \{0,\pm 1,\pm2\}$ in
Proposition~\ref{diff_sec} and more precisely
\begin{equation} \label{besser}
d(x;w) 	= \begin{cases}
	\lvert\Delta(x;w)\rvert &\mbox{if }\lvert\Delta(x;w)\rvert\in [0,\pi], \\
	\lvert\Delta(x;w) - 2\pi \sigma\rvert  = 2\pi -  \lvert\Delta(x;w)\rvert &\mbox{if}  \; \;\lvert\Delta(x;w)\rvert \in (\pi,2 \pi],\\
	\lvert\Delta(x;w)\rvert - 2\pi &\mbox{if } \lvert\Delta(x;w)\rvert \in (2\pi,3 \pi],\\
	\lvert\Delta(x;w) - 4\pi \sigma\rvert  = 4\pi -  \lvert\Delta(x;w)\rvert&\mbox{if}  \; \; \lvert\Delta(x;w)\rvert \in (3\pi,4 \pi),
	\end{cases}
\end{equation}
where $\sigma = \sgn\left(\Delta(x;w) \right) \in \{-1,1\}$ and
\[
\sgn (x) :=
\left\{
\begin{array}{rl}
1 & \text{ if} \; x >0,\\
0 & \text{ if} \; x = 0,\\
-1& \text{ otherwise.}
\end{array}
\right.
\]
\begin{proof}
	Since $\lvert x_j - x_k\rvert < 2 \pi$ for $x_j,x_k \in [-\pi,\pi)$,  we see
	that $\lvert\Delta_2(x)\rvert < 4 \pi$ and
	$\lvert\Delta_{1,1}(x)\rvert < 4 \pi$.
	
	First we consider $d_2$. By Lemma~\ref{cyc_diff_perm} we obtain
	\begin{equation}\label{d2_expl}
		d_{2}(x)
		= \min_{k \in \I_3} \bigl\lvert\Delta(P^\tT x^k;b_2)\bigr\rvert
		= \min_{k \in \I_3} \Bigl\lvert\Delta_{2}(x) + 2 \pi \bigl\langle \sum_{j=1}^2 e_j, P b_2 \bigr\rangle \Bigr\rvert,
	\end{equation}
	where we can assume by the cyclic shift invariance of $d_2$ that $x_{j_1} = x_1$.

	If $x^1= (x_1, x_2, x_3)^\tT$, then the corresponding permutation matrix $P$ in Lemma~\ref{cyc_diff_perm} is the identity matrix.
	Further we obtain that $\Delta_{2}(x) = (x_1-x_2)+(x_3-x_2) \in (-2\pi,2\pi)$ and by~\eqref{d2_expl} we get
	\begin{equation*}
		\lvert\Delta_{2}(P^\tT x^2)\rvert
		= \lvert\Delta_{2}(x^2)\rvert
		= \lvert\Delta_{2}(x)+2\pi\rvert
		\qquad\text{and}\qquad
		\lvert\Delta_{2}(x^3)\rvert
		= \lvert\Delta_{2}(x)-2\pi\rvert.
	\end{equation*}
	If $x^1= (x_1, x_3, x_2)^\tT$, then $P = \left(\begin{smallmatrix} 1&0&0\\0&0&1\\0&1&0\end{smallmatrix}\right)$
	and $\Delta_{2}(x)\in (-4\pi,0]$.
	In this case we get
	\begin{equation*}
		\lvert\Delta_{2}(P^\tT x^2)\rvert
		= \lvert\Delta_{2}(x)+2\pi\rvert
		\qquad \text{and}\qquad
		\lvert\Delta_{2}(P^\tT x^3)\rvert
		= \lvert\Delta_{2}(x)+4\pi\rvert.
	\end{equation*}
	This proves the first assertion.

	For \(d_{1,1}\) we can again assume that \(x_{j_1} = x_1\).
	Exploiting that \[\Delta_{1,1}(x_1,x_2,x_3,x_4) = \Delta_{1,1}(x_1,x_3,x_2,x_4)\]
	we have to consider the following three cases:

	If $x^1= (x_1, x_2, x_3,x_4)^\tT$, then $P$ is the identity matrix,
	$\Delta_{1,1}(x) = (x_2-x_1) + (x_3 - x_4) \in (-2\pi,2\pi)$ and
	\begin{equation*}
		\lvert\Delta_{1,1}( x^2)\rvert =  \lvert\Delta_{2}(x)-2\pi\rvert, \quad
		\lvert\Delta_{1,1}( x^3)\rvert =  \lvert\Delta_{2}(x)\rvert, \quad
		\lvert\Delta_{1,1}( x^4)\rvert =  \lvert\Delta_{2}(x)+2\pi\rvert.
	\end{equation*}
	If $x^1= (x_1, x_2, x_4, x_3)^\tT$,
	then~$P = \left(\begin{smallmatrix} 1&0&0&0\\0&1&0&0\\0&0&0&1\\0&0&1&0 \end{smallmatrix}\right)$
	and $\Delta_{2}(x)\in [0,2\pi)$. By ~\eqref{helpful} we have
	\begin{equation*}
		\lvert\Delta_{1,1}( x^2)\rvert = \lvert\Delta_{2}(x)-2\pi\rvert, \quad
		\lvert\Delta_{1,1}( x^3)\rvert = \lvert\Delta_{2}(x)\rvert, \quad
		\lvert\Delta_{1,1}( x^4)\rvert = \lvert\Delta_{2}(x)-2\pi\rvert.
	\end{equation*}
	If $x^1= (x_1, x_4, x_2, x_3)^\tT$,
	then $P = \left(\begin{smallmatrix} 1&0&0&0\\0&0&0&1\\0&1&0&0\\0&0&1&0 \end{smallmatrix}\right)$
	and $\Delta_{2}(x)\in [0,4\pi)$. Here we obtain
	\begin{equation*}
		\lvert\Delta_{1,1}( x^2)\rvert = \lvert\Delta_{2}(x)-2\pi\rvert, \quad
		\lvert\Delta_{1,1}( x^3)\rvert = \lvert\Delta_{2}(x)-4\pi\rvert, \quad
		\lvert\Delta_{1,1}( x^4)\rvert = \lvert\Delta_{2}(x)-2\pi\rvert.
	\end{equation*}
	This finishes the proof.
\end{proof}
%
\section{Proximal mapping of absolute cyclic differences} \label{sec:prox}
%
For a proper, closed, convex function $\varphi: \R^N \rightarrow (-\infty,+\infty]$
and $\lambda > 0$ the {\itshape proximal mapping} $\prox_{\lambda \varphi}: \R^N \rightarrow \R^N$
is defined by
\begin{equation} \label{prox_R}
	\prox_{\lambda \varphi}(f)
	:= \argmin_{x \in \R^N} \frac{1}{2} \lVert f-x\rVert_2^2 + \lambda \varphi (x),
\end{equation}
see~\cite{Mor62}.
The above minimizer exits and is uniquely determined.
Many algorithms which were recently used in variational image processing reduce
to the iterative computation of values of proximal mappings.
An overview of applications of proximal mappings is given in~\cite{PB2013}.

In this section, we are interested in proximal mappings of absolute cyclic
differences \(d(\cdot;w)^p\), i.e.,
\(
\prox_{\lambda d(\cdot;w)^p}:(\mathbb S^1)^d \rightarrow (\mathbb S^1)^d,
\)
for $w \in \mathbb R^d$. More precisely, we will determine
for ${\mathbb S}^1$-valued vectors represented by $f \in [-\pi,\pi)^d$
the values
\[
\prox_{\lambda d(\cdot;w)^p }(f)
:= \argmin_{x \in [-\pi,\pi)^d } \frac{1}{2} \sum_{j=1}^d d(x_j,f_j) ^2
	+ \lambda d(x;w)^p, \qquad \lambda > 0
\]
for $p \in \{ 1,2\}$ and first and second order absolute cyclic differences
$d(\cdot;w)$, $w \in \{b_1,b_2,b_{1,1}\}$.
Here $\argmin_{x \in [-\pi,\pi)^d }$ means that we are looking for the representative of $x \in ({\mathbb S}^1)^d$
in $ [-\pi,\pi)^d$.
In particular, we will see that
these proximal mapping are single-valued for $f \in [-\pi,\pi)^d$ with
$\lvert(\langle f,w \rangle)_{2\pi} \rvert < \pi$ and have two values for
$\lvert(\langle f,w \rangle)_{2\pi} \rvert = \pi$.

We start by considering the proximal mappings of the appropriate differences
in $\R^d$. Then we use the results to find the proximal functions of the
absolute cyclic differences.
%
\subsection{Proximity of differences on \texorpdfstring{$\R^d$}{ℝ\textsuperscript{d}}}
First we give analytical expressions for
$\prox_{\lambda \lvert\langle\cdot,w\rangle - a\rvert^p}$,
where $p \in \{1,2\}$ and $w \in \R^d$, $a \in \R$.
Since we could not find a corresponding reference in the literature, the
computation of the minimizer of
\begin{equation} \label{prox_1}
	E (x;f,a,w) := \frac{1}{2} \lVert f - x\rVert_2 ^2 + \lambda \lvert \langle x,w \rangle - a\rvert^p,
	 \qquad \lambda > 0
\end{equation}
is described in the following lemmas.
We start with $p=1$.

\begin{lemma} \label{lem:proxy_R}
For given $f \in \mathbb R^d$ and $0 \not = w \in \mathbb R^d$, $a \in \R$
set
\[
s := \sgn (\langle f,w \rangle - a) \quad \text{and} \quad
\mu := \frac{\langle f,w \rangle - a}{\lVert w\rVert_2^2}.
\]
Then the minimizer $\hat x$ of
\begin{equation} \label{prox}
	E (x;f,a,w) :=
	\frac{1}{2} \lVert f - x\rVert_2 ^2 + \lambda \lvert \langle x,w \rangle
		- a\rvert, \qquad \lambda > 0
\end{equation}
is given by
\begin{equation} \label{argmin_E}
\hat x = f - s \, \min\{\lambda ,\lvert\mu\rvert\} \,  w
\end{equation}
and the minimum by
\begin{equation} \label{min_E}
E(\hat x;f,a,w) =
\left\{
\begin{array}{cl}
	\lVert w\rVert_2^2 \, \frac{1}{2} \mu^2
		&\text{ if } \lvert\mu\rvert  \le \lambda,\\
\lVert w\rVert_2^2 \bigl( \frac{1}{2} \lambda^2
		+ \lambda (\lvert\mu\rvert - \lambda )\bigr)&\text{ otherwise.}
\end{array}
\right.
\end{equation}
\end{lemma}
\begin{proof}
Since $w \not = 0$, there exists a component $w_j \not = 0$
and we rewrite
\[
	E (x;f,a,w)
	= \frac{1}{2}\lVert f - x\rVert_2 ^2 + \lambda \lvert w_j\rvert \, \big\lvert \big\langle \frac{w}{w_j}, x - \frac{a}{w_j} e_j \big\rangle \big\rvert.
\]
Substituting $y := x - \frac{a}{w_j} e_j$, $g = f - \frac{a}{w_j} e_j$
and~$\nu := \lambda \lvert w_j\rvert$, $v := \frac{w}{w_j}$ we see
that~$\hat x = \hat y$, where~$\hat y$ is the minimizer of
\[
	F (y;g,v) := \frac{1}{2} \lVert g - y\rVert_2 ^2
		+ \nu \lvert\langle v,y \rangle\rvert.
\]
The (Fenchel) dual problem of  $\argmin_{y \in \mathbb R^d} F(y)$ reads
\begin{equation} \label{dual}
	\hat t := \argmin_{t \in \mathbb R} \left\{ \lVert g - t \,  w \rVert_2^2
		\quad \mbox{subject to} \quad \lvert t\rvert \le \nu \right\}
\end{equation}
and the relation between the minimizers of the primal and dual problems is given by
\begin{equation} \label{pd}
 \hat y = g - \hat t \, v.
\end{equation}
Rewriting~\eqref{dual} we see that $\hat t$ is the minimizer of
\[
(t - \tilde \mu)^2 \quad \mbox{subject to} \quad \lvert t \rvert \le \nu,
\]
where
$\tilde \mu := \frac{\langle v,g \rangle}{\lVert v \rVert^2}$.
Hence we obtain
\[
\hat t =
\left\{
\begin{array}{cl}
\tilde \mu  &\text{ if } \lvert \tilde \mu \rvert  \le \nu,\\
\sgn(\tilde \mu) \nu &\text{ otherwise.}
\end{array}
\right.
\]
and by~\eqref{pd} further
\[
\hat y = g - \sgn(\tilde \mu) \min \{\nu, \lvert \tilde \mu \rvert\} \, v.
\]
Substituting back results in~\eqref{argmin_E} and plugging $\hat x$ into $E$ we get~\eqref{min_E}.
\end{proof}
%
\begin{example} \label{ex1} Let $p=1$, $a = 0$, and $E (x;f,w) := E (x;f,0,w)$.
\begin{enumerate}
	\item
	For $w = b_1 = (-1, 1)^\tT$ and $f \in \mathbb R^2$ we get
$\lVert w\rVert_2^2 = 2$ and $s = \sgn(f_2 - f_1)$ so that the minimizer of $E(x;f,b_1)$
follows by {\itshape soft shrinkage} of $f$ with threshold $\lambda$:
\begin{equation} \label{second_m}
\hat x = \left( \begin{array}{l} f_1 + s \, m  \\ f_2 - s \, m  \end{array} \right) ,\qquad m := \min\{\lambda, \frac{\lvert f_2 - f_1\rvert} {2}\}.
\end{equation}
	\item
	For $w = b_2 = (1, -2, 1)^\tT$ and $f \in \mathbb R^3$
we obtain $\lVert w\rVert_2^2 = 6$ and $s = \sgn (f_1 - 2f_2 + f_3)$.
Consequently, the minimizer of $E(x;f,b_2)$
is given by
\begin{equation} \label{second_m1}
\hat x = \left( \begin{array}{l} f_1 - s \, m  \\ f_2 + 2s \, m \\ f_3 - s \, m \end{array} \right),
\qquad m := \min\left\{\lambda, \frac{\lvert f_1 - 2f_2 + f_3\rvert} {6}\right\}.
\end{equation}
\item
For $w = b_{1,1} = (-1,1, 1,-1)^\tT$ and $f \in \mathbb R^4$
we obtain $\lVert w\rVert_2^2 = 4$ and $s = \sgn(f_2 - f_1 + f_3-f_4)$, so that the minimizer of
$
E(x;f,b_{1,1}
$
is given by
\begin{equation} \label{second_m2}
\hat x = \left( \begin{array}{l} f_1 + s\, m  \\ f_2 - s\, m \\ f_3 - s\, m \\f_4 + s\, m \end{array} \right),
\qquad m := \min\left\{\lambda, \frac{\lvert f_2 - f_1 + f_3-f_4\rvert} {4}\right\}.
\end{equation}
\end{enumerate} 
\end{example}
We will apply the following corollary.
\begin{corollary} \label{diff_data}
Let $0 \not = w \in \mathbb R^d$. Further, let $f,\tilde f \in \mathbb R^d$ and $a, \tilde a \in \R$
be given such that
$\lvert\langle f,w \rangle -a\rvert < \lvert\langle \tilde f,w \rangle - \tilde a\rvert$.
Then
\begin{equation} \label{ass}
\min_{x \in \mathbb R^d} E(x;f,a,w) < \min_{x \in \mathbb R^d} E(x;\tilde f,\tilde a,w).
\end{equation}
\end{corollary}
\begin{proof}
Set
$\mu := \frac{\langle f,w \rangle - a}{\lVert w\rVert_2^2}$ and
$\tilde \mu := \frac{\langle \tilde f,w \rangle- \tilde a}{\lVert w\rVert_2^2}$.
By assumption $\lvert\mu\rvert < \lvert\tilde  \mu\rvert$
and according to~\eqref{min_E} we have to consider three cases.
\begin{enumerate}[label=\arabic*.]
	\item 
	Let $\lvert\tilde \mu\rvert \le \lambda$. Then by assumption
	also~$\lvert \mu\rvert < \lambda$
	and we conclude by~\eqref{min_E} that
\[
\min_{x \in \mathbb R^d} E(x;f,a,w)
= \frac{1}{2} \lVert w\rVert^2_2 \mu^2 < \frac{1}{2} \lVert w\rVert_2^2 \tilde \mu^2
= \min_{x \in \mathbb R^d} E(x;\tilde f,\tilde a,w).
\]
	\item 
Let $\lvert\tilde \mu\rvert > \lambda$ and $\lvert \mu\rvert \le \lambda$.
By~\eqref{min_E} this implies
\begin{align*}
\min_{x \in \mathbb R^d} E(x; f,a,w)
	&= \frac{1}{2} \lVert w\rVert_2^2 \mu^2,\\
\min_{x \in \mathbb R^d} E(x;\tilde f,\tilde a,w)
	&= \frac{1}{2} \lVert w\rVert_2^2 \lambda^2 + \lVert w\rVert_2^2 \lambda (\lvert \tilde \mu \rvert - \lambda).
\end{align*}
Since $\lVert w\rVert^2 \lambda (\lvert \tilde \mu \rvert - \lambda) > 0$ and $\lvert \mu\rvert \le \lambda$
we obtain
$\min_{x \in \mathbb R^d} E(x; f,a,w) < \min_{x \in \mathbb R^d} E(x;\tilde f,\tilde a,w)$.
\item 
Let $\lvert\tilde \mu\rvert > \lambda$ and $\lvert \mu\rvert > \lambda$.  By~\eqref{min_E} this
implies
\begin{align*}
    \min_{x \in \mathbb R^d} E(x; f,a,w)
&= \frac{1}{2} \lVert w\rVert_2^2 \lambda^2 + \lVert w\rVert_2^2 \lambda( \lvert\mu\rvert -  \lambda) \\
&< \frac{1}{2} \lVert w\rVert_2^2 \lambda^2 + \lVert w\rVert_2^2 \lambda( \lvert\tilde\mu\rvert -  \lambda)
= \min_{x \in \mathbb R^d} E(x; \tilde f,\tilde a,w)
\end{align*}
and  we are done.\qedhere
\end{enumerate}
\end{proof}
Next we consider the case $p=2$.
\begin{lemma}\label{lem:E_quad}
Let $0 \not = w\in \R^d$.
\begin{enumerate}
 \item
Then, for $f \in \R^d$ and $a\in \R$, the minimizer $\hat{x}$ of
\begin{equation}
 E(x;f,a,w)= \norm{f-x}{2}^{2} + \lambda \bigl(\langle x,w \rangle -a\bigr)^{2}\text{,}\qquad \lambda > 0\label{E_quad}
\end{equation}
is given by
\begin{equation*}
\hat{x} = f-\frac{\lambda (\langle f,w \rangle - a)}{1+\lambda \norm{w}{2}^{2}} \, w
\end{equation*}
and the minimum by
\begin{equation} \label{quad_min}
 E(\hat{x};f,a,w)= \frac{\lambda }{1+\lambda \norm{w}{2}^{2}}\bigl(\langle f,w \rangle -a\bigr)^{2} .
\end{equation}
\item If $\bigl(\langle f,w \rangle -a\bigr)^{2}<\bigl(\langle \tilde{f},w \rangle -\tilde{a}\bigr)^{2}$ for some
$f, \tilde f \in \R^d$ and $a, \tilde a\in \R$, then
\begin{equation}
 \underset{x\in\R^d}{\min} E(x;f,a,w) < \underset{x\in\R^d}{\min}E(x;\tilde{f},\tilde{a},w).
\end{equation}
\end{enumerate}
\end{lemma}
%
\begin{proof}
\begin{enumerate}
	\item Setting the gradient of~\eqref{E_quad} to zero results in
\begin{align*}
 2(x-f) + 2\lambda (\langle x,w\rangle-a)\, w &= 0,\\
 (I+\lambda w w^\tT)x &= f + \lambda a w.
\end{align*}
Using the Sherman-Morrison formula~\cite[p. 129]{Bj96} it follows
\begin{align*}
 \hat{x}  &=\left(I-\frac{\lambda}{1+\lambda \norm{w}{2}^{2}}w w^\tT\right)(f+\lambda a w)\\
	  & = f -\frac{\lambda \langle f,w\rangle}{1+\lambda \norm{w}{2}^{2}} \, w + \lambda a w - \frac{\lambda^{2} a \norm{w}{2}^{2}}{1+\lambda \norm{w}{2}^{2}} \, w\\
	  & = f- \frac{\lambda (\langle f,w \rangle- a)}{1+\lambda \norm{w}{2}^{2}} \, w.
\end{align*}
For the corresponding energy we obtain by straightforward computation
\begin{align*}
  E(\hat{x};f,a,w)
  & = \norm{f-\hat{x}}{2}^{2} + \lambda \bigl(\langle x,w \rangle -a\bigr)^{2}\\
 & = \frac{\lambda^{2}\left(\langle f,w\rangle -a\right)^{2}}{\left(1+\lambda \norm{w}{2}^{2}\right)^{2}} \norm{w}{2}^{2}+
\lambda \left[\langle f,w\rangle - \frac{\lambda\left(\langle f,w\rangle -a\right)\lVert w\rVert_2^2}{1+\lambda \norm{w}{2}^{2}}-a \right]^{2}\\
  & = \frac{\lambda }{1+\lambda \norm{w}{2}^{2}}\bigl(\langle f,w \rangle -a\bigr)^{2} \text{.}
\end{align*}
	\item follows directly from~\eqref{quad_min}.\qedhere
\end{enumerate}
\end{proof}


\subsection{Proximity of absolute cyclic differences of first and second order}
Now we turn to $\mathbb S^1$-valued data represented by
$f \in [-\pi,\pi)^d$.
We are interested in the minimizers of
\begin{equation} \label{cp}
{\mathcal E}(x;f,w) := \frac{1}{2}  \sum_{j=1}^d  d(f_j,x_j)^2  + \lambda d(x;w)^p, \quad \lambda > 0
\end{equation}
on $[-\pi,\pi)^d$ for $p \in \{1,2\}$ and $w \in \{ b_1,b_2,b_{1,1} \}$.
We start with  the case $p=1$.
%
\begin{theorem} \label{lem:proxy_b1}
For $w \in \{b_1,b_2,b_{1,1}\}$ set $s := \sgn(\langle f,w \rangle)_{2 \pi}$.
Let $p=1$ and $f \in [-\pi,\pi)^d$, where $d$ is adapted to the respective
length of $w$.
\begin{enumerate}
\item
If $\lvert(\langle f,w \rangle)_{2\pi} \rvert < \pi$, then the unique minimizer
of~${\mathcal E}(x;f,w)$ is given by
\begin{equation} \label{first_e}
\hat x =  (f - s \, m \,w)_{2\pi}, \qquad
m:= \min \left\{\lambda ,\frac{\lvert (\langle f,w \rangle)_{2\pi} \rvert} {\lVert w\rVert_2^2} \right\}.
\end{equation}
\item
\label{proxy-casepi}
If $\lvert(\langle f,w \rangle)_{2\pi} \rvert = \pi$,
then ${\mathcal E}(x;f,w)$ has the two minimizers
\begin{equation} \label{first_ee}
\hat x =  (f \mp s \, m \, w)_{2\pi},
 \qquad m:= \min\left\{\lambda ,\frac{\pi}{\lVert w\rVert_2^2} \right\}.
\end{equation}
\end{enumerate}
\end{theorem}
Note that for $w = b_1$ case~\ref{proxy-casepi}) appears exactly if $f_1$ and
$f_2$ are antipodal points.
\begin{proof}
By~\eqref{d} and Lemma~\ref{diff_sec} we can rewrite ${\mathcal E}$ in~\eqref{cp} as
\begin{align}
	{\mathcal E}(x;f,w)
	&:= \frac{1}{2}\sum_{j=1}^d \min_{k_j \in \Z}\lvert f_j - x_j - 2\pi k_j\rvert^2 + \lambda \min_{\sigma \in \Z} \lvert\langle x,w \rangle - 2\pi \sigma\rvert \\
&= \min_{\atop{k \in \Z^d}{\sigma \in \Z} }  \frac{1}{2} \lVert f - x - 2\pi k\rVert_2^2 + \lambda \lvert\langle x,w \rangle - 2\pi \sigma\rvert,
\end{align}
where $k = (k_j)_{j=1}^d$.
Let
\[
E_{k,\sigma}(x) := \frac{1}{2} \lVert f - x - 2\pi k\rVert_2^2 + \lambda \lvert\langle x,w \rangle - 2\pi \sigma\rvert.
\]
We are looking for
\begin{equation} \label{aha}
\min_{x \in [-\pi,\pi)^d} {\mathcal E}(x;f,w)
=  \min_{x \in [-\pi,\pi)^d} \min_{\atop{k \in \Z^d}{\sigma \in \Z}  } E_{k,\sigma}(x)
= \min_{\atop{k \in \Z^d}{\sigma \in \Z}  } \min_{x \in [-\pi,\pi]^d} E_{k,\sigma}(x),
\end{equation}
where the last equality can be seen by the following argument:
If for some $k,\sigma$ the minimizer
\(\hat x := \argmin_{x \in [-\pi,\pi]^d} E_{k,\sigma}(x)\)
has components \(\hat x_j = \pi\) for \(j \in J \subseteq \I_d\), then we get
using \(\tilde x := \hat x - 2\pi \sum_{j \in J} e_j \in [-\pi,\pi)^d\), that
\[
E_{k,\sigma} (\hat x)
= \frac{1}{2} \lVert f - \tilde x - 2\pi\underbrace{( k - \sum_{j \in J} e_j )}_{\tilde k}\rVert_2^2
+ \lambda \lvert\langle \tilde x,w \rangle - 2\pi \underbrace{(\sigma - \langle \sum_{j \in J} e_j,w \rangle)}_{\tilde \sigma}\rvert
=
E_{\tilde k, \tilde \sigma} (\tilde x).
\]
By Lemma~\ref{lem:proxy_R} the minimizers over $\R^d$ of $E_{k,\sigma}(x)$
are given by
\begin{equation} \label{important}
 \hat x_{k,\sigma} = f - 2\pi k - s_{k,\sigma} \, m_{k,\sigma} \, w,
\end{equation}
where
\begin{equation}
s_{k,\sigma} := \sgn \left( \nu_{k,\sigma} \right) , \quad
m_{k,\sigma} := \min \left\{ \lambda , \frac{\lvert  \nu_{k,\sigma} \rvert} {\lVert w\rVert_2^2} \right\}
\quad \text{and} \quad
\nu_{k,\sigma} := \langle f,w \rangle - 2\pi ( \langle k,w \rangle + \sigma).
\end{equation}
By Corollary~\ref{diff_data} the minimum of $E_{k,\sigma}$ is determined
by~$\lvert\nu_{k,\sigma}\rvert$. Note that 
$\lvert\langle f,w \rangle\rvert < 2 \pi$
for~$w = b_1$ and $\lvert\langle f,w \rangle\rvert < 4 \pi$
for~$w \in \{b_2,b_{1,1}\}$. We distinguish two cases.
\begin{enumerate}[label=\arabic*.]
\item 
If $\langle f,w \rangle \in  ( (2r-1)\pi, (2r+1)\pi )$, $r \in \Z$
then $\nu_{k,\sigma}$ attains its smallest value exactly for $\langle k,w \rangle + \sigma = r$ and
\[
\nu_{k,r - \langle k,w \rangle} = \langle f,w \rangle - 2\pi r  =  (\langle f,w \rangle)_{2\pi}.
\]
By~\eqref{important} we obtain
\[
\hat x_{k,r - \langle k,w \rangle} = f - 2 \pi k - s \, m \, w
\]
with $s,m$ as in~\eqref{first_e}.
Corollary~\ref{diff_data} implies that
\[
E_{k,r - \langle k,w \rangle}(\hat x_{k,r - \langle k,w \rangle}) < E_{k,\sigma} (\hat x_{k,\sigma})
\le \min_{x \in [-\pi,\pi]^d} E_{k,\sigma} (x)\qquad \forall \sigma \in \Z \backslash \{r - \langle k,w \rangle \}.
\]
Finally, there exists exactly one $k^*\in \Z^d$ such that
$\hat x_{k^*,r - \langle k^*,w \rangle} \in [-\pi,\pi)^d$ and
by~\eqref{aha} we conclude that
\[
	\hat x := \hat x_{k^*,r - \langle k^*,w \rangle} = f - 2 \pi k^* - s \, m \, w = (f  - s \, m \, w)_{2\pi}
\]
is the unique minimizer of ${\mathcal E}(x;f,w)$ over $[-\pi,\pi)^d$.
\item
If $\langle f,w \rangle = (2r-1)\pi$, $r \in \Z$,
then $\nu_{k,\sigma}$ attains its smallest value exactly for
$\langle k,w \rangle + \sigma \in \{r,r-1\}$ and by Corollary~\ref{diff_data}
the minimum of the corresponding functions
$E_{k,\sigma}$ is smaller than those of the other functions in~\eqref{aha}.
We obtain
\begin{equation*}
	\nu_{k,r - \langle k,w \rangle} = -\pi, \quad \nu_{k,r-1 - \langle k,w \rangle} = \pi	
\end{equation*}
and
\begin{equation*}
	\hat x_{k,r - \langle k,w \rangle} = f - 2 \pi k +  m \, w , \quad \hat x_{k,r-1 - \langle k,w \rangle}
= f - 2 \pi k -  m \, w, \qquad m:= \min\left\{\lambda ,\frac{\pi}{\lVert w\rVert_2^2} \right\}.
\end{equation*}
\end{enumerate}
As in part 1 of the proof we conclude that
$
\hat x = (f \pm m \, w)_{2\pi}
$
are the minimizers of ${\mathcal E}(x;f,w)$ over $[-\pi,\pi)^d$. This finishes the proof.
\end{proof}
Next we focus on $p = 2$.
\begin{theorem} \label{thm:p2}
Let $p=2$ in~\eqref{cp}, $w \in \{b_1,b_2,b_{1,1}\}$ and $f \in [-\pi,\pi)^d$, where $d$ is adapted to the respective length of $w$.
\begin{enumerate}
\item
If $\lvert(\langle f,w \rangle)_{2\pi} \rvert < \pi$, then the unique minimizer of
${\mathcal E}(x;f,w)$ is given by
\begin{equation}
\hat x = \left(f-\frac{\lambda (\langle f,w \rangle)_{2\pi}}{1+\lambda \norm{w}{2}^{2}} w\right)_{2 \pi}\text{.}
\end{equation}
\item
If $\lvert(\langle f,w \rangle)_{2\pi} \rvert = \pi$,
then ${\mathcal E}(x;f,w)$ has the two minimizers
\begin{equation}
\hat x = \left(f \mp \frac{\lambda \pi}{1+\lambda \norm{w}{2}^{2}} w \right)_{2 \pi}\text{.}
\end{equation}
\end{enumerate}
\end{theorem}
\begin{proof}
	The proof follows the lines of the proof of Theorem~\ref{lem:proxy_b1} using Lemma~\ref{lem:E_quad}.
\end{proof}
%
Finally, we need the proximal mapping ${\prox}_{\lambda d(f,\cdot)^{2}}$ for given $f \in (\mathbb S^1)^N$.
The proximal mapping of the (squared) cyclic distance function was also computed (for more general manifolds) in~\cite{FO02}.
Here we give an explicit expression for spherical data.
%
\begin{proposition} \label{Theo:Prox_quad}
For $f,g \in [-\pi,\pi)^N$ let
\begin{equation}
	{\mathcal E}(x;g,f) := d(g,x)^{2} + \lambda d(f,x)^{2} =
	\sum_{j=1}^{N} d(g_j,x_j)^{2}+\lambda d(f_j,x_j)^{2}.
\end{equation}
Then the minimizer(s) of
${\mathcal E}(x;g,f)$ are given by
\begin{equation} \label{min_quad_1}
\hat x =  \left( \frac{g+\lambda f}{1+\lambda} + \frac{\lambda}{1+\lambda} \, 2\pi \, v \right)_{2\pi},
\end{equation}
where $v = (v_j)_{j=1} ^N \in \R^N$ is defined by
\[
v_j :=
\left\{
\begin{array}{ll}
0 & \mbox{if \(\lvert g_j-f_j\rvert \le \pi \)},\\
\sgn(g_j - f_j) & \mbox{ if \(\lvert g_j-f_j\rvert > \pi\)}
 \end{array}
\right.
\]
and the minimum is
\begin{equation}
\mathcal{E}(\hat x;g,f) = \frac{\lambda}{1+\lambda}(g-f)_{2\pi}^{2}.
\end{equation}
\end{proposition}
\begin{proof}
Obviously, the minimization of ${\mathcal E}$ can be done component wise so that we can restrict our attention to $N=1$.
\begin{enumerate}[label=\arabic*.]
	\item 
	First we look at the minimization problem over $\R$ which reads
\begin{equation}
\min_{x\in\mathbb R}\ (g-x)^{2} + \lambda (f-x)^{2}
\end{equation}
and has the following minimizer and minimum:
\begin{equation}
 \hat{x}  = \frac{g+\lambda f}{1+\lambda}, \qquad (g-\hat{x})^{2} + \lambda(f-\hat{x})^{2} = \frac{\lambda}{1+\lambda}(g-f)^{2}.
\end{equation}
\item 
For the original problem
\begin{align*}
\underset{x\in [-\pi,\pi)}{\min}{\mathcal E}(x;g,f)
& =
\underset{x\in [-\pi,\pi)}{\min} \Bigl\{d(g,x)^{2}+\lambda d(f,x)^{2}\Bigr\} \\
 & = \min_{x\in [-\pi,\pi)} \left\{\min_{k\in \{0,\sgn (g)\}} (g-x-2\pi k)^{2} + \min_{l\in \{0,\sgn (f) \}} \lambda (f-x-2\pi l)^{2}\right\}
\end{align*}
we consider the related energy functionals on $\R$, namely
\begin{equation}
   E_{k,l}(x;g,f) := (g-x-2\pi k)^{2} + \lambda (f-x-2\pi l)^{2}, \quad k\in \{0,\sgn g\}, \, l\in \{0,\sgn f \}.
\end{equation}
By part 1 of the proof these functions have the minimizers
\begin{equation*}
 \hat{x}_{k,l} =  \frac{(g-2\pi k)+\lambda (f-2\pi l)}{1+\lambda} = \frac{g+\lambda f -2 \pi (k+\lambda l)}{1+\lambda }
\end{equation*}
and
\begin{equation}\label{Prox_quad}
   E_{k,l}(\hat{x}_{k,l};g,f) = \frac{\lambda}{1+\lambda}\left( (g-2\pi k)-(f-2\pi l)\right)^{2}
	= \frac{\lambda}{1+\lambda}\left( g-f-2\pi(k- l)\right)^{2}.
\end{equation}
We distinguish three cases:
\begin{enumerate}
 \item If $\abs{g-f}< \pi$, then
the minimum in~\eqref{Prox_quad} occurs exactly for $k=l$ and it holds
\begin{equation*}
 \hat{x}_{k,k} = \frac{g+\lambda f -2 \pi k(1+\lambda )}{1+\lambda }
= \frac{g+\lambda f}{1+\lambda} -2\pi k.
\end{equation*}
For $k=0$ we see that $\hat{x}_{0,0} \in [-\pi,\pi)$ and
$
\mathcal{E}(\hat x;g,f) = \frac{\lambda}{1+\lambda}(g-f)^{2}.
$
\item If $\abs{g-f}> \pi$, then~\eqref{Prox_quad} has its minimum exactly for $k-l = \sgn(g-f)$ and
	\begin{equation*}
		\hat{x}_{k,k-\sgn(g-f)} = \frac{g+\lambda f -2 \pi (k+\lambda(k-\sgn(g-f)))}{1+\lambda}
		= \frac{g+\lambda (f+\sgn(g-f) 2\pi)}{1+\lambda} -2\pi k
	\end{equation*}
which is in $[-\pi,\pi)$ for $k=0$ or $k=\sgn(g)$ and
	\begin{equation*}
		\mathcal{E}(\hat x;g,f) = \frac{\lambda}{1+\lambda}(g-f- \sgn(g-f) 2\pi)^{2}.
	\end{equation*}
\item In the case $\abs{g-f}= \pi$ the minimum in~\eqref{Prox_quad} is attained
for $k-l=0,\pm 1$ so that we have both solutions from i) and ii).
This completes the proof.\qedhere
\end{enumerate}
\end{enumerate}
\end{proof}
\section{Cyclic proximal point method} \label{sec:cpp}
%
The proximal point algorithm (PPA) on the Euclidean space goes back to~\cite{Roc76}.
Recently this algorithm was extended to Riemannian manifolds of non-positive
sectional curvature~\cite{FO02} and also to Hadamard spaces~\cite{Bac13}.
A cyclic version of the proximal point algorithm (CPPA) on the Euclidean space
was given in~\cite{Ber11}, see also the survey~\cite{Ber10}.
A CPPA for Hadamard spaces can be found
in~\cite{Bac13a}.
In the CPPA the original function $J$ is split into a sum $J = \sum_l J_l$ and,
iteratively, the proximal mappings of the functions $J_l$ are applied in a
cyclic way. The great advantage of this method is that often the proximal
mappings of the summands $J_l$ are much easier to compute or can even be given
in a closed form.
In the following we develop a CPPA for functionals of $\mathbb S^1$-valued signals and images containing
absolute cyclic first and second order differences.
%
\subsection{One-dimensional data} \label{sec:signals}
First we  have a look at the one-dimensional case, i.e., at signals.
For given $\mathbb S^1$-valued signals represented by
\(f = \bigl(f_i)_{i=1}^N \in[-\pi,\pi)^N \), \(N \in\mathbb N\),
and regularization parameters \(\alpha,\beta \ge 0\), \(\max\{\alpha, \beta\} \neq 0\), we are
interested in
\begin{align} \label{task_1}
	\argmin_{x\in[-\pi,\pi)^N} J(x), \quad
	J(x) = J(x,f) := F(x; f)
	+ \alpha  \operatorname{TV}_1(x)
	+ \beta \operatorname{TV}_2(x),	
\end{align}
where
\begin{align}\label{eq:Fdist}
	F(x; f) & := \frac{1}{2} \sum_{i=1}^{N} d(f_i,x_i)^2,\\
	\operatorname{TV}_1(x) & := \sum_{i=1}^{N-1}d(x_i,x_{i+1}), \quad
	\operatorname{TV}_2(x)  := \sum_{i=2}^{N-1} d_2(x_{i-1},x_i,x_{i+1}).\label{eq:TV2}
\end{align}
To apply a CPPA we set $J_1(x) := F(x; f)$,
split  \(\alpha \operatorname{TV}_1\)
into an even and an odd part
\begin{equation} \label{eq:splitTVevodd}
	\alpha\operatorname{TV}_1(x)
	=
\sum_{\nu=0}^1
	\alpha\sum_{i=1}^{\bigl\lfloor\!\frac{N-1}{2}\!\bigr\rfloor}
		 d(x_{2i-1+\nu},x_{2i-\nu})
	=: \sum_{\nu=0}^1 J_{2+\nu}(x)
\end{equation}
and \(\beta \operatorname{TV}_2\) into three sums
\begin{align*}
\beta \operatorname{TV}_2(x)
=
\sum_{\nu=0}^2
		\beta\sum_{i=1}^{\bigl\lfloor\!\frac{N-1}{3}\!\bigr\rfloor}
			d_2(x_{3i-2+\nu},x_{3i-1+\nu},x_{3i+\nu})
	=: \sum_{\nu=0}^2 J_{4+\nu}(x)
\end{align*}
Then  the objective function decomposes as
\begin{equation}
J = \sum_{l=1}^6 J_l.
\end{equation}
We compute in the $k$-th cycle of the CPPA the signal
\begin{equation}
x^{(k)} := \prox_{\lambda_k J_6} \left( \prox_{\lambda_k J_5} \ldots \left( \prox_{\lambda_k J_1}  (x^{(k-1)}) \right) \right).
\end{equation}
The different proximal values can be obtained as follows:
\begin{enumerate}
 \item
By Proposition~\ref{Theo:Prox_quad} with $x^{(k-1)}$ playing the role of $g$ we get
\begin{equation}
	x^{(k-1 + \frac{1}{6})} := \prox_{\lambda_k J_{1}}(x^{(k-1)}).
\end{equation}
\item
For $\nu=0,1$, we obtain the vectors
\[
	x^{(k-1 + \frac{\nu+2}{6})} := \prox_{\lambda_k J_{2+\nu}} \left( x^{(k-1 + \frac{\nu+1}{6})} \right)
\]
by applying Theorem~\ref{lem:proxy_b1} with $w = b_1$ independently for the pairs $(x_{2i-1 + \nu},x_{2i + \nu})$, $i = 1,\ldots,\bigl\lfloor\!\frac{N-1}{2}\!\bigr\rfloor$.
\item
For $\nu=0,1,2$, we compute
\[
x^{(k-1 + \frac{\nu+4}{6})} := \prox_{\lambda_k J_{4+\nu}} \left( x^{(k-1 + \frac{\nu+3}{6})} \right)
\]
by applying Theorem~\ref{lem:proxy_b1} with $w = b_2$ independently for the vectors $(x_{3i-2 +\nu},x_{3i -1 + \nu},x_{3i + \nu} )$, $i = 1,\ldots,\bigl\lfloor\!\frac{N-1}{3}\!\bigr\rfloor$.
\end{enumerate}
The parameter sequence $\{\lambda_k\}_k$ of the algorithm should fulfill
\begin{equation}\label{eq:CPPAlambda}
		\sum_{k=0}^\infty \lambda_k = \infty,
		\quad \text{and}
		\quad \sum_{k=0}^\infty \lambda_k^2 < \infty.
\end{equation}
This property is also essential for proving the convergence of the CPPA for real-valued data
and data on a Hadamard manifold, see~\cite{Bac13a,Ber11}.
In our numerical experiments we choose $\lambda_k := \lambda_0/k$ with some initial parameter $\lambda_0 >0$
which clearly fulfill~\eqref{eq:CPPAlambda}. The whole procedure is summarized in Algorithm~\ref{alg:CPPA}.
%
\begin{algorithm}[tbp]
	\caption[]{\label{alg:CPPA} CPPA for minimizing~\eqref{task_1} or~\eqref{task_2} for cyclic data}
	\begin{algorithmic}
		\State \textbf{Input} $\{ \lambda_k \}_k$ fulfilling~\eqref{eq:CPPAlambda}
		and $\alpha$,  $\beta$ or $\alpha = (\alpha_1,\alpha_2)$, $\beta = (\beta_1,\beta_2)$, $\gamma$
		\State data \(f\in[-\pi,\pi)^N\) or \(f\in[-\pi,\pi)^{N\times M}\)\\\vspace{-.5\baselineskip}
		\Function {CPPA}{$\alpha$, $\beta$, $\lambda_0$, $f$}
		\State Initialize \(x^{(0)}=f\), \(k=0\)
		\State Initialize the cycle length as \(c = 6\) (1D) or \(c=15\) (2D)
		\Repeat
		\For{$l \gets 1$ \textbf{to} $c$}
			\State \(x^{(k+\frac{l}{c})} \leftarrow \prox_{\lambda_k J_l}(x^{(k+\frac{l-1}{c})})\)
		\EndFor
		\State $k \gets k+1$
		\Until a convergence criterion are reached
		\State\Return $x^{(k)}$
		\EndFunction
	\end{algorithmic}
\end{algorithm}
%
\subsection{Two-dimensional data}\label{sec:2DTV}
Next we consider two-dimensional data, i.e., images of the
form~\(f := \bigl(f_{i,j})_{i,j=1}^{N,M}\in[-\pi,\pi)^{N\times M}\), $N,M \in \mathbb N$.
Our functional includes \emph{h}orizontal and \emph{v}ertical cyclic first and second order differences $d_1$ and $d_2$
and \emph{d}iagonal (mixed) differences $d_{1,1}$.
For non-negative regularization parameters
$\alpha := (\alpha_1, \alpha_2)$,
$\beta:= (\beta_1,\beta_2)$ and $\gamma$
not all equal to zero
we are looking for
\begin{align}\label{task_2}
\argmin_{x\in[-\pi,\pi)^{N\times M}} J(x),\quad
J(x) = J(x,f) :=
	F(x; f)
	+ \alpha \operatorname{TV}_1(x)
	+ \beta \operatorname{TV}_2^{\mathrm{hv}}(x)
	+ \gamma \operatorname{TV}_2^{\mathrm{d}}(x),
\end{align}
where
\begin{align}
	F(x;f)
	&:=
		\frac{1}{2}
		\sum_{i,j=1}^{n,m} d(f_{i,j},x_{i,j})^2,\label{eq:2DFdist}
\\
	\alpha\operatorname{TV}_1(x)
	&:=
		\alpha_1\sum_{i,j=1}^{N-1,M} d(x_{i,j},x_{i+1,j})
		+
		\alpha_2\sum_{i,j=1}^{N,M-1} d(x_{i,j},x_{i,j+1}),
	\label{eq:2DTV}
\\
	\beta\operatorname{TV}_2^{\mathrm{hv}}(x)
	&:=
	\beta_1\sum_{i=1,j=2}^{N-1,M}
	d_2(x_{i-1,j},x_{i,j},x_{i+1,j}),
	\beta_2\sum_{i=2,j=1}^{N,M-1}
	d_2(x_{i,j-1},x_{i,j},x_{i,j+1})+
	\label{eq:TV2iso}\\
	\gamma \operatorname{TV}_2^{\mathrm{d}}(x)
	&:=
	\gamma \sum_{i,j=1}^{N-1,M-1}
	d_{1,1}(x_{i,j},x_{i+1,j},x_{i,j+1},x_{i+1,j+1}).
	\label{eq:TV2mix}	
\end{align}
Here the objective function splits as
\begin{equation} \label{splitting_15}
J = \sum_{l=1}^{15} J_l
\end{equation}
with the following summands:
Again we set $J_1 := F(x;f)$ and compute
the proximal value of $\lambda_k J_1$ by Proposition~\ref{Theo:Prox_quad}.
Each of the sums in
\(\operatorname{TV}_1\)
and~\(\operatorname{TV}_2^{\text{hv}}\)
can be split analogously as in the
one-dimensional case,
where we have to consider row and column vectors now.
This results in $2(2+3) = 10$ functions $J_2, \ldots,J_{11}$ whose proximal
values can be computed by Theorem~\ref{lem:proxy_b1}.
Finally, we split \(\operatorname{TV}_2^{\text{d}}\).
into the four sums
\begin{equation}\label{eq:TV2mixsplitting}
	\gamma \operatorname{TV}_2^{\text{d}} (x) =
	\sum_{\mu,\nu=0}^1\gamma
	\sum_{i,j=1}^{\bigl\lfloor\frac{N-1}{2}\bigr\rfloor,\bigl\lfloor\frac{M-1}{2}\bigr\rfloor}
	d_{1,1}(x_{2i-1+\mu,2j-1+\nu},x_{2i+\mu,2j-1+\nu},x_{2i-1+\mu,2j+\nu},x_{2i+\mu,2j+\nu})
\end{equation}
and denote the inner sums by $J_{12}, \ldots,J_{15}$. Clearly, the proximal
values of the functions $\lambda_k J_l$, $l=12,\ldots,15$ can be computed
separately for the vectors
\[
	(x_{2i-1+k,2j-1+l},x_{2i+k,2j-1+l},x_{2i-1+k,2j+l},x_{2i+k,2j+l}),
\; i=1,\ldots\bigl\lfloor\frac{N-1}{2}\bigr\rfloor, \;
j=1,\ldots,\bigl\lfloor\frac{M-1}{2}\bigr\rfloor
\]
by Theorem~\ref{lem:proxy_b1} with $w = b_{1,1}$.
In summary, the computation can be done by Algorithm~\ref{alg:CPPA}.
Note that the presented approach immediately generalizes to arbitrary
dimensions.
%
\subsection{Convergence} \label{sec:conv}
%
Since \(\mathbb S^1\) is not a Hadamard space, the convergence analysis of the
CPPA in~\cite{Bac13a} cannot be applied.
We show the convergence of the CPPA for the 2D $\mathbb S^1$-valued
function~\eqref{task_2} under certain conditions.
The 1D setting in~\eqref{task_1} can then be considered as a special case.
In the following, let $\mathbb I := \{1,\ldots,N\}\times\{1,\ldots,M\}$.

Our first condition is that the data $f \in (\mathbb S^1)^{N\times M}$
is dense enough, this means that the distance between neighboring pixels
\begin{equation} \label{eq:defSupNormOfDiff}
d_\infty(f)
: = \max_{(i,j) \in \mathbb I}
		\max_{(k,l) \in N_{i,j}} d(f_{i,j},f_{k,l}),
		\qquad
		\mathcal N_{i,j}
		:=
	\bigl\{
		(k,l)\in\mathbb I
		:
		\,
		\lvert i-k\rvert+\lvert l-j\rvert=1
	\bigr\}
\end{equation}
is sufficiently small.
Similar conditions also appear in the convergence analysis of nonlinear
subdivision schemes for manifold-valued data
in~\cite{WallnerDynCAGD,WeinmannConstrApprox}.
In the context of nonlinear subdivision schemes, even more severe restrictions
such as `almost equally spaced data' are frequently
required~\cite{OswaldHarizanovNormal}.
This imposes additional conditions on the second
order differences to make the data almost lie on a `line'. Our analysis
requires only bounds on the first, but not on the second order differences.

Our next requirement is that the regularization
parameters~$\alpha, \beta, \gamma$ in~\eqref{task_2} are sufficiently small.
For large parameters any solution tends to become almost constant.
In this case, if the data is for example
equidistantly distributed on the circle, e.g.,
$f_{i} = 2\pi i/N$ in 1D,  any~$2\pi j/N$ shift
is again a solution.
In this situation the
model loses its interpretation which is an inherent problem due to
the cyclic structure of the data.

Finally, the parameter sequence $\{\lambda_k\}_k$ of the CPPA has to
fulfill~\eqref{eq:CPPAlambda} with a small $\ell^2$ norm. The later
can be achieved by rescaling.

Our convergence analysis is based on a convergence result in~\cite{Bac13a} and
an unwrapping procedure. We start by reformulating  the convergence result for
the CPPA of real-valued data, which is a special case of~\cite{Bac13a} and
can also be derived from~\cite{Ber10}.
%
\begin{theorem} \label{thm:bacak}
Let $E = \sum_{l=1} ^c E_l$, where $E_l$, $l=1,\ldots,c$, are proper, closed, convex functionals
on $\mathbb R^{N \times M}$. Let $E$ have a global minimizer.
Assume that there exists $L >0$ such that the iterates $\{x^{(k+\frac{l}{c})} \}$ of the CPPA (see Algorithm~\ref{alg:CPPA})
satisfy
\[
E_l(x^{(k)}) - E_l (x^{(k+\frac{l}{c})}) \le L \| x^{(k)} - x^{(k+\frac{l}{c})}\|_2, \quad l=1,\ldots,c,
\]
for all $k\in \mathbb N_0$.
Then the sequence $\{x^{(k)} \}_k$ converges to a minimizer of $E$.
Moreover the iterates fulfill
\begin{align} \label{noch_2}
&\lVert x^{(k+\frac{l-1}{c})} - x^{(k+\frac{l}{c})}\lVert_2 \leq 2 \lambda_k L,
\\
   &\lVert x^{(k+1)}-x\rVert_2^2
	\leq
	\lVert x^{(k)}-x\rVert_2^2
		- 2 \lambda_k [E(x^{(k)})-E(x)]
		+ 2 \lambda_k^2 L^2 c(c+1) \quad \mbox{for all} \; x \in \mathbb R^{N \times M}.\label{noch_1}
		\end{align}
\end{theorem}
%
The next lemma states a discrete analogue of a well-known
result on unwrapping or
lifting from algebraic topology. We supply a short proof since we did not found
it in the literature.
\begin{lemma}\label{lem:discreteCovering}
Let $x \in (\mathbb S^1)^{N \times M}$ with $d_\infty(x) < \frac{\pi}{2}$.
For $q \in \mathbb S^1$ not antipodal to $x_{1,1}$ fix
an~$\tilde x_{1,1} \in \mathbb R$ such
that~\(\exp_q(\tilde x_{1,1} ) = x_{1,1}\).
Then there exists a unique \(\tilde x \in\mathbb R^{N\times M}\) such that for
all~$(i,j) \in \mathbb I$ the following relations are fulfilled:
\begin{enumerate}
	\item \( \exp_q(\tilde x_{i,j}) = x_{i,j} , \)
	\item \(d(x_{i,j},x_{k,l}) = \lvert\tilde x_{i,j} - \tilde x_{k,l}\rvert,
			\quad (k,l) \in \mathcal N_{i,j}\).
	\end{enumerate}
 We call $\tilde{x}$ the {\rm lifted} or {\rm unwrapped} image of $x$ (w.r.t.\ a fixed $\tilde x_{1,1}$).
\end{lemma}
\begin{proof}
For~\(x_{k,l}\), \((k,l)\in\mathcal N_{1,1}\),
it holds by assumption on $d_\infty(x)$ that~\(d(x_{1,1},x_{k,l}) < \frac{\pi}{2}\).
Hence we have
\(s_{k,l} := (x_{k,l}-x_{1,1})_{2\pi}
	\in
	\bigl(-\frac{\pi}{2},\frac{\pi}{2}\bigr)
	\),
	where with an abuse of notation $x_{k,l}$ stands
	for an arbitrary representative in $T_q \mathbb S^1$
	of $x_{k,l}$.
	Then obviously
\( \tilde x_{k,l} := \tilde x_{1,1} + \sgn (s_{k,l})d(x_{1,1},x_{k,l})\) , $(k,l) \in N_{1,1}$
are the unique values satisfying i) and ii).

For \(x_{2,2}\in\mathcal N_{2,1} \cap \mathcal N_{1,2}\) consider
\begin{align*}
	\tilde x_{2,2}
	&:= \tilde x_{1,2} + \sgn \bigl( (x_{2,2}-x_{1,2})_{2\pi} \bigr)d(x_{1,2},x_{2,2})\\
	&\ = \tilde x_{1,1} + \sgn (s_{1,2})d(x_{1,1},x_{1,2})
		+ \sgn \bigl( (x_{2,2}-x_{1,2})_{2\pi} \bigr)d(x_{1,2},x_{2,2}),
	\\
	\tilde y_{2,2}	
	&:= \tilde x_{2,1} + \sgn \bigl( (x_{2,2}-x_{2,1})_{2\pi} \bigr)d(x_{2,1},x_{2,2})\\
	&\ = \tilde x_{1,1} + \sgn (s_{2,1})d(x_{1,1},x_{2,1})
		+ \sgn \bigl( (x_{2,2}-x_{2,1})_{2\pi} \bigr)d(x_{2,1},x_{2,2}).
\end{align*}
By assumption on $d_\infty(x)$ we see that $|\tilde x_{2,2} - \tilde y_{2,2}| < 2 \pi$
so that $\tilde x_{2,2}  = \tilde y_{2,2} $. Thus $\tilde x_{2,2}$
is the unique value with properties i) and ii).

Proceeding this scheme successively, we obtain
the whole unique image \(\tilde x\) fulfilling i) and ii).
\end{proof}
%
For \(\delta\in (0,\pi)\) we define
\begin{equation}\label{eq:DefM}
		{\mathcal S}(f,\delta)
	:=
	\bigl\{
	x \in (\mathbb S^1)^{N\times M}
	:\,
	d_{\infty}(x,f) \leq \delta
	\bigr\},
\end{equation}
where 
\begin{equation}
d_{\infty}(x,f) := \max_{(i,j) \in \mathbb I} d(x_{i,j},f_{i,j}),
\end{equation}
to measure how `near' the images $f$ and $x$ are to each other.
%
\begin{lemma} \label{lem:ConvLem2}
Let $f\in (\mathbb S^1)^{N \times M}$ with~\( d_\infty(f) < \frac{\pi}{8}\)
and $q \in \mathbb S^1$ be not antipodal to $f_{1,1}$. Fix $\tilde f_{1,1}$ with
$\exp_q(\tilde f_{1,1}) = f_{1,1}$ and let $\tilde f$ be the corresponding lifting of $f$.
Let $\delta \in (0,\frac{\pi}{8}]$.
\begin{enumerate}
	\item
Then every $x \in {\mathcal S}(f,\delta)$ has a unique lifting~$\tilde{x}$
w.r.t.\ to the base point $q$ with $|\tilde x_{1,1} - \tilde f_{1,1}| \le \frac{\pi}{8}$.
	\item
For \( J \) defined by ~\eqref{task_2}, let  
$\tilde J$ denote its analog for real-valued data, i.e.,
\begin{equation} \label{tilde_J}
\tilde J(x) = \tilde J(x,\tilde f) :=
	\widetilde{F}(x;\tilde f)
	+ \alpha \widetilde{ \operatorname{TV}}_1(x)
	+ \beta \widetilde{ \operatorname{TV}}_2^{\mathrm{hv}}(x)
	+ \gamma \widetilde{ \operatorname{TV}}_2^{\mathrm{d}}(x),
\end{equation}
where the cyclic distances in $F$ and  in the~$\operatorname{TV}$ terms are replaced by
absolute differences in $\tilde F$ and $\widetilde{ \operatorname{TV}}$.
Then it holds
\begin{equation} \label{eq:LiftJ}
J(x)= \tilde{J}(\tilde{x}) \quad \mbox{for all} \quad x \in {\mathcal S}(f,\delta).
\end{equation}
\end{enumerate}
\end{lemma}
%
\begin{proof}
By definition of ${\mathcal S}(f,\delta)$ and assumption on $f$
we have for any~$x \in {\mathcal S}(f,\delta)$
that
\begin{equation*}
	d(x_{i,j}, x_{k,l})
	\le
	d(x_{i,j}, f_{i,j}) + d(f_{i,j}, f_{k,l}) + d(f_{k,l}, x_{k,l})
	< \frac{3\pi}{8}
	,\quad (k,l) \in {\mathcal N}_{i,j},
\end{equation*}
and hence
$d_\infty(x)< \frac{3\pi}{8}$. Further it holds~$d(x_{1,1},f_{1,1}) < \frac{\pi}{8}$.
Consequently, every $x \in {\mathcal S}(f,\delta)$ has a unique lifting~$\tilde{x}$
by Lemma~\ref{lem:discreteCovering}
w.r.t.\ to the base point $q$ fulfilling $|\tilde x_{1,1} - \tilde f_{1,1}| \le \frac{\pi}{8}$.

To see~\eqref{eq:LiftJ} we show the equality for the involved summands in~$J$
and~$\tilde J$ separately.

First we consider $\operatorname{TV}_1$.
By properties of the lifting in Lemma~\ref{lem:discreteCovering} we have
$d(x_{i,j},x_{i,j+1})= |\tilde{x}_{i,j}-\tilde{x}_{i,j+1}|$
and
$d(x_{i,j},x_{i+1,j})= |\tilde{x}_{i,j}-\tilde{x}_{i+1,j}|$.
By the definition of $\operatorname{TV}_1$
and~$\operatorname{\widetilde{TV}}_1$, this
implies~$\operatorname{TV}_1(x) = \operatorname{\widetilde{TV}}_1(\tilde{x})$.

Next we consider $\operatorname{TV}_2^{\mathrm{hv}}$. The corresponding
second order differences are given by the
expressions~$d_2(x_{i-1,j},x_{i,j},x_{i+1,j})$
and~$d_2(x_{i,j-1},x_{i,j},x_{i,j+1})$, respectively.
We exemplarily consider the first term.
Since
\[
	d(x_{i-1,j},x_{i+1,j})
	\leq
	d(x_{i-1,j},f_{i-1,j}) + d(f_{i-1,j},f_{i+1,j}) + d(f_{i+1,j},x_{i+1,j})
	< \frac{\pi}{8} + \frac{\pi}{4} + \frac{\pi}{8} = \frac{\pi}{2}
\]
the distance between any two members of the triple is
smaller than $\frac{\pi}{2}$.
Due to the properties of the lifting~\(\tilde x\) this implies
\(|\Delta(\tilde x_{i-1,i},\tilde x_{i,j}, \tilde x_{i+1,j}; b_2)| < \pi\).
Then we conclude by Proposition~\ref{diff_sec} that
$\operatorname{TV}^{\mathrm{hv}}_2(x) = \operatorname{\widetilde{TV}}^{\mathrm{hv}}_2(\tilde{x})$.
Similarly it follows that \( \operatorname{TV}^{\mathrm{d}}_2(x)
	= \operatorname{\widetilde{TV}}^{\mathrm{d}}_2(\tilde{x})\).

Concerning the data term \(F(x;f)\) we consider \(e_{i,j} := d(x_{i,j},f_{i,j})\)
and \(\tilde e_{i,j} := \lvert \tilde x_{i,j} - \tilde f_{i,j}\rvert\).
By definition of ${\mathcal S}(f,\delta)$
we have
\(e_{i,j} \le \delta = \frac{\pi}{8}\)
and by construction of
\(\tilde f\) and \(\tilde x\) that \(\tilde e_{i,j} = e_{i,j} + 2\pi k_{i,j}\),
\(k_{i,j}\in\mathbb N\) and \(k_{1,1} = 0\). Furthermore it holds
\( |\tilde e_{i,j+1}-\tilde e_{i,j}|
=
\bigl|
	 | \tilde x_{i,j+1} - \tilde f_{i,j+1}|
	 -
	 | \tilde x_{i,j} - \tilde f_{i,j}|
\bigr| \le 2\delta\).
If \(k_{i,j}\neq k_{i,j+1}\), then there exists
\(k\in\mathbb Z\backslash\{0\}\)
such that
\[
|\tilde e_{i,j+1}-\tilde e_{i,j}|
= |\tilde e_{i,j+1}-\tilde e_{i,j} + 2\pi k| \geq 2\pi-2\delta > 2\delta
\]
which is a contradiction. Thus $k_{i,j}= k_{i,j+1}$. Similarly we conclude
$k_{i,j}= k_{i+1,j}$.
In summary we obtain
\(k_{i,j} = k_{1,1} = 0\) for all \((i,j)\in\mathbb I\)
which implies
~\(e_{i,j}=\tilde e_{i,j}\). This finishes the proof.
\end{proof}
\begin{remark} \label{rem:Conv}
The set ${\mathcal S}(f,\delta)$ is a convex subset of
$(\mathbb S^1)^{N\times M}$
which means that for $x,y \in {\mathcal S}(f,\delta)$ and $t \in [0,1]$
we have $[x,y]_t \in {\mathcal S}(f,\delta)$. Here $[x,y]_t$ denotes the point
reached after time t on the unit speed geodesic starting at $x$ in direction
of~$y$.
Recall that a function $\varphi$ is convex on ${\mathcal S}(f,\delta)$ if for
all~$x,y \in {\mathcal S}(f,\delta)$
and all $\lambda \in [0,1]$ the relation
$\varphi([x,y]_t) \le t\varphi(x) + (1-t) \varphi(y)$
holds true.
Let $f\in (\mathbb S^1)^{N \times M}$ with~\( d_\infty(f) < \frac{\pi}{8}\)
and~$\delta \in (0,\frac{\pi}{8}]$.
Then we conclude by Lemma~\ref{lem:ConvLem2}, since $\tilde J$ is convex,
that \( J \)  is convex on \({\mathcal S}(f,\delta) \).
\end{remark}
\begin{lemma} \label{lem:ConvLem1}
Let $f \in (\mathbb{S}^1)^{N \times M}$
and
\(m := \max\{\alpha_1,\alpha_2,\beta_1,\beta_2,\gamma\} >0\).
Let $\varepsilon>0$ such that
\begin{align}
\label{eq:EstAlpaBeta}
	\operatorname{TV}_1(f)
	+\operatorname{TV}_2^{\mathrm{hv}}(f)
	+\operatorname{TV}_2^{\mathrm{d}}(f)
	\leq
	\frac{\varepsilon^2}{m}\text{.}
\end{align}
Then any minimizer~$x^\ast$ of $J$ in~\eqref{task_2} fulfills
\begin{equation}
\label{eq:ResultLemma1}
	d_{\infty}(x^\ast,f) 
	\le \left( \sum_{i,j=1}^{N,M} d(x_{i,j}^*,f_{i,j})^2 \right)^\frac{1}{2}
	\leq \varepsilon.
\end{equation}
\end{lemma}
%
\begin{proof}
Any minimizer \(x^*\) of~\eqref{task_2} satisfies
\begin{align}
	J(x^*)
	\leq J(f)
	& \leq m \,
	\bigl(\operatorname{TV}_1(f)
		+\operatorname{TV}_2^{\mathrm{hv}}(f)
		+\operatorname{TV}_2^{\mathrm{d}}(f)
	\bigr).
\end{align}
As a consequence we obtain
\begin{equation*}
	d_{\infty}(x^*,f)^2
	\leq
	\sum_{i,j=1}^{N,M} d(x_{i,j}^*,f_{i,j})^2
	\leq
	m
	\bigl(\operatorname{TV}_1(f)
		+\operatorname{TV}_2^{\mathrm{hv}}(f)
		+\operatorname{TV}_2^{\mathrm{d}}(f)
	\bigr)
	\leq
	\varepsilon^2.\qedhere
\end{equation*}
\end{proof}
\begin{remark}\label{rem:TildeJ}
	Lemma~\ref{lem:ConvLem1} holds also true for real-valued
	data and $\tilde J$ in~\eqref{tilde_J}.
\end{remark}
%
Now we combine Lemma~\ref{lem:ConvLem1} and~\ref{lem:ConvLem2} to locate
the minimizers of $J$ and $\tilde J$.
%
\begin{lemma} \label{lem:ConvLem3}
Let $f \in (\mathbb S^1)^{N \times M}$
with~$d_\infty(f) < \frac{\pi}{8}$
and~$0<\varepsilon<\delta\leq\frac{\pi}{8}$ be given.
Choose the parameters $\alpha,\beta,\gamma$ of $J$
in~\eqref{task_2} such that~\eqref{eq:EstAlpaBeta} with $\varepsilon$  holds true.
Then any minimizer $x^\ast$ of~$J$ lies in \({\mathcal S}(f,\delta) \).
Furthermore, if $\tilde{f}$ is the unique lifting of $f$ w.r.t.\ a base
point~$q$ and fixed $\tilde{f}_{1,1}$ with $\exp_q (\tilde{f}_{1,1}) = f_{1,1}$,
then each minimizer~$y^\ast$ of~$\tilde{J}$ defines a
minimizer $x^\ast:=\exp_q(y^\ast)$ of~$J$.
Conversely, the uniquely defined lifting $\tilde{x}^\ast$ of a
minimizer~$x^\ast$ of $J$ is a minimizer of $\tilde{J}$.
\end{lemma}
\begin{proof}
By Lemma~\ref{lem:ConvLem1} we obtain
~$d_\infty(x^\ast,f) \leq \varepsilon< \frac{\pi}{8}$
so that $x^\ast \in {\mathcal S}(f,\delta)$.

In order to show the second statement note that the
mapping~$x \mapsto \tilde{x}$ is a bijection from \({\mathcal S}(f,\delta) \) to the set \({\widetilde {\mathcal S}}(f,\delta) \)
defined by
\begin{equation*}
	{\widetilde {\mathcal S}}(f,\delta)
	:=
	\bigtimes\limits_{(i,j)\in\mathbb I}
		\bigl[\tilde{f}_{i,j}-\delta,\tilde{f}_{i,j}+\delta\bigr].
\end{equation*}
If $y^\ast$ minimizes $\tilde{J}$, then it lies in ${\widetilde {\mathcal S}}(f,\delta)$ 
which follows by Remark~\ref{rem:TildeJ}. By~\eqref{eq:LiftJ}
and the minimizing property of $y^\ast$ we obtain for any $x \in {\mathcal S}(f,\delta)$ that
\begin{equation*}
J (\exp_q(y^\ast)) = \tilde{J}(y^\ast) \leq \tilde{J}(\tilde{x}) = J(x).
\end{equation*}
As a consequence, $\exp_q(y^\ast)$ is a
minimizer of $J$ on ${\mathcal S}(f,\delta)$. By Lemma~\ref{lem:ConvLem1} all the minimizers of $J$ are
contained in ${\mathcal S}(f,\delta)$ so that  $\exp_q(y^\ast)$ is a
minimizer of $J$ on $(\mathbb S^1)^{N\times M}$.

We proceed with the last statement. Let $x^\ast$ be  a minimizer of $J$ with lifting $\tilde{x}^\ast$.
Then we get for
any~$\tilde y \in \widetilde{S}(f,\delta)$ that
\begin{equation}
  \tilde{J}(\tilde{x}^\ast)  =  J (x^\ast) \leq  J(\exp_q(\tilde y)) = \tilde{J}(\tilde y).
\end{equation}
This shows that $\tilde{x}^\ast$ is a minimizer of $\tilde{J}$ on ${\widetilde {\mathcal S}}(f,\delta)$.
Since by Remark~\ref{rem:TildeJ} all minimizers of $\tilde{J}$ lie
in~${\widetilde {\mathcal S}}(f,\delta)$, the last assertion follows.
\end{proof}
Next we locate the iterates of the CPPA for real-valued data on a ball whose
radius can be controlled.
%
\begin{lemma} \label{lem:ConvLem4}
For $f\in \mathbb R^{N\times M}$ and $\lambda := \{\lambda_k\}_k$ with property~\eqref{eq:CPPAlambda},
let $\bigl\{ x^{(k+\frac{l}{c})} \bigr\}$ be the sequence
produced by Algorithm~\ref{alg:CPPA} for $\tilde{J}$.
Assume that \(\lVert f - x^{(k+\frac{l}{c})}\rVert_{\infty} \leq \pi\).
Let $x^\ast\in\mathbb R^{N \times M}$ be the minimizer of
$\tilde{J}$.
Then, for $k \in \mathbb N _0$ and $l \in \{1,\ldots,c\}$, it holds
\begin{align}
\label{eq:estCPPAiterates}
	\lVert x^{(k+\frac{l}{c})}-x^\ast\rVert_2
	\leq R :=
	\sqrt{ \lVert f-x^\ast\rVert_2^2 + 2 \lVert \lambda\rVert_2^2 L^2 c (c+1)} + 2 \lVert\lambda\rVert_\infty cL,
\end{align}
where $c=15$ denotes the number of inner iterations and $L=4$.
\end{lemma}

The assumption on the distances \(|f_{i,j} - x^{(k+\frac{l}{c})}_{i,j}|\), $(i,j) \in \mathbb I$, to be
smaller than \(\pi\) is automatically
fulfilled for any unwrapping of \(\mathbb S^1\)-valued data.
\begin{proof}
By Theorem~\ref{thm:bacak} we know that
\begin{equation}\label{eq:appliedThm41}
   \lVert x^{(k+1)}-x\rVert_2^2
	\leq
	\lVert x^{(k)}-x\rVert_2^2
		- 2 \lambda_k [\tilde{J}(x^{(k)})-\tilde{J}(x)]
		+ 2 \lambda_k^2 L^2 c(c+1).
\end{equation}
As a constant $L$ we can choose the maximum of the Lipschitz constants of the
involved summands.
For \(\operatorname{\widetilde{TV}}_1\),
\(\operatorname{\widetilde{TV}}_2^{\mathrm{hv}}\)
and~\(\operatorname{\widetilde{TV}}_2^{\mathrm{d}}\) the Lipschitz constants
are~\(1\), \(4\), and \(4\), respectively. For the quadratic data term we
have
\[
\frac{1}{2} \left| |f_{i,j} -x_{i,j}|^2 -|f_{i,j}-y_{i,j}|^2 \right|
	\le \frac12 |2  f_{i,j} -x_{i,j}-y_{i,j}| |x_{i,j}-y_{i,j}|
	\le \pi  |x_{i,j}-y_{i,j}|.
\]
Therefore, we can set $L = 4$.
Plugging in the minimizer~$x = x^\ast$ into~\eqref{eq:appliedThm41} and using $x^{(0)}=f$ yields
\begin{align}
	\lVert x^{(k+1)}-x^\ast\rVert_2^2 
	&\leq  \lVert x^{(k)}-x^\ast\rVert_2^2 + 2 \lambda_k^2 L^2 c(c+1)
	\\
	&\leq  \lVert x^{(0)}-x^\ast\rVert_2^2 + 2 \sum_{j=0}^k \lambda_j^2 L^2 c(c+1)\\
	&\leq  \lVert f-x^\ast\rVert_2^2 + 2 \lVert \lambda \rVert_2^2 L^2 c(c+1).\label{eq:EstCycle}
\end{align}
By Theorem~\ref{thm:bacak} it holds
\begin{equation} \label{eq:EstIntermediate}
   \lVert x^{(k+\frac{l}{c})} - x^{(k+\frac{l-1}{c})}\lVert_2 \leq 2 \lambda_k L.
\end{equation}
Using the triangle inequality we obtain
\[
\lVert x^{(k+\frac{l}{c})}-x^\ast\rVert_2 \le
\lVert x^{(k+\frac{l}{c})} - x^{(k+\frac{l-1}{c})}\lVert_2 + \ldots + \lVert x^{(k+\frac{1}{c})} - x^{(k)}\lVert_2 +
\lVert x^{(k)}-x^\ast\rVert_2^2
,
\]
which implies the assertion by~\eqref{eq:EstCycle} and~\eqref{eq:EstIntermediate}.
\end{proof}
Now we compare the proximal mappings 
acting on data with values in~$\mathbb{S}^1$ and $\R.$
%
\begin{lemma}\label{lem:ConvLem5}
For \(f\in(\mathbb S^1)^{N\times M}\) with~$d_\infty(f) < \frac{\pi}{8}$,
let $J$ be defined by~\eqref{task_2} with the splitting~\eqref{splitting_15}.
Let $\tilde{f}$ be the unique lifting of $f$ w.r.t.\ a base
point~$q$ not antipodal to $f_{1,1}$ and fixed $\tilde{f}_{1,1}$ with $\exp_q (\tilde{f}_{1,1}) = f_{1,1}$.
Further, denote by $\tilde J$  the functional~\eqref{tilde_J} corresponding to $J$.
Then, for any $x \in \mathcal S(f,\delta)$, \(\delta \in (0, \frac{\pi}{8}]\)
and its lifting~$\tilde x$ w.r.t.\ $q$, we have
\begin{equation} \label{eq:commuteProxProj}
  \prox_{\lambda J_l}(x) = \exp_q(\prox_{\lambda \tilde{J}_l}(\tilde{x})), \quad l \in \{1,\ldots,15\},
\end{equation}
i.e., the canonical projection $\exp_q$ commutes with the proximal mappings.
\end{lemma}
\begin{proof}
The function $J_1$ is based on the distance to the data $f$.
Since $x \in \mathcal S(f,\delta)$, we have
$d(x_{i,j},f_{i,j}) \leq \frac{\pi}{8}$ for all $(i,j) \in \mathbb I$.
The components of the
proximal mapping $\prox_{\lambda J_1}$ are given by
Proposition~\ref{Theo:Prox_quad} from which we
conclude~\eqref{eq:commuteProxProj} for $l=1.$

The proximal mappings of $J_l$, $l=2,\ldots,15$, are given via proximal mappings of
the first and second order cyclic differences.  We consider the first
order difference $d_1=d$.
By the triangle inequality, we
have~$d(x_{i,j},x_{i,j+1}) \leq \frac{3\pi}{8}$ as well
as~$d(x_{i,j},x_{i+1,j}) \leq \frac{3\pi}{8}$.
By the explicit form of the proximal
mapping in Theorem~\ref{lem:proxy_b1} we obtain~\eqref{eq:commuteProxProj} 
for~$J_l$, $l=2,\ldots,5$.

Next we consider the horizontal and vertical second order
differences~$d_2(x_{i-1,j},x_{i,j},x_{i+1,j})$
and~$d_2(x_{i,j-1},x_{i,j},x_{i,j+1})$. We have
that~$d(x_{i,j-1},x_{i,j}) < \frac{3\pi}{8}$,
$d(x_{i,j},x_{i,j+1})<\frac{3\pi}{8}$ as well
as~$d(x_{i,j-1},x_{i,j}) < \frac{\pi}{2}$. Hence all contributing values
of~\(x\) lie on a quarter of the circle. Applying the proximal mapping in
Theorem~\ref{lem:proxy_b1} the resulting data lie on one half of the circle.
An analogous statement holds true for the horizontal part.
Hence the proximal mappings of the ordinary second differences agree with
the cyclic version (under identification via $\exp_q$). This
implies~\eqref{eq:commuteProxProj} for $J_l$, $l=6,\ldots,11$.
 
Finally, we consider the mixed second order
differences~$d_{1,1}(x_{i,j},x_{i+1,j},x_{i,j+1},x_{i+1,j+1})$.
As above, we have for neighboring data items that the distance is smaller
than~$\frac{3\pi}{8}$.
For all four contributing values of \(x\) we have that the
pairwise distance is smaller by $\frac{\pi}{2}$. Thus again they lie on a
quarter of the circle. Hence, the proximal mapping for the ordinary mixed second
differences agree with the cyclic version (under identification via $\exp_q$). This
implies~\eqref{eq:commuteProxProj} for $J_l$, $l=12,\ldots,15$.
\end{proof}

We note that Lemma~\ref{lem:ConvLem5} does not guarantee
that~$\prox_{\lambda J_l}(x)$ remains in $\mathcal S(f,\delta)$. Therefore it does not
allow for an iterated application. In the following main theorem we combine the preceding
lemmas to establish this property.
%
\begin{theorem} \label{thm:Convergence}
Let $f\in(\mathbb S^1)^{N\times M}$ with
$d_{\infty}(f) < \frac{\pi}{8}$.
Let
$\lambda := \{\lambda_k\}_k$ fulfill property~\eqref{eq:CPPAlambda}
and
\begin{align} \label{eq:condTheorem}
	 \sqrt{ \varepsilon^2 + 2 \lVert \lambda\rVert_2^2 L^2 c (c+1)}
	+ 2 \lVert\lambda\rVert_\infty cL  < \frac{\pi}{16},
\end{align}
for some  \(\varepsilon>0\), where $c = 15$ and $L=4$.
Further, assume that the parameters~$\alpha,\beta,\gamma$
of the functional $J$ in~\eqref{task_2} and $\varepsilon$ satisfy~\eqref{eq:EstAlpaBeta}.
Then the sequence $\{x^{(k)} \}_k$ generated by the
CPPA in Algorithm~\ref{alg:CPPA} converges to a global minimizer of $J$. 
\end{theorem}
%
\begin{proof} 	
Let $\tilde{f}$ be the lifting of of $f$ with respect to a base point $q$ not antipodal to $f_{1,1}$
and fixed $\tilde f_{1,1}$ with $\exp_q (\tilde{f}_{1,1}) = f_{1,1}$.
Further, let $\tilde J$ denote the real analog of $J$.
By Lemma~\ref{lem:ConvLem2} we have
$\operatorname{TV}_1(f) = \operatorname{\widetilde{TV}}_1(\tilde{f})$
and
$\operatorname{TV}_2^\bullet(f) = \operatorname{\widetilde{TV}}_2^\bullet(\tilde{f})$ for
$\bullet \in \{{\rm hv}, \rm{d} \}$
such that 
\eqref{eq:EstAlpaBeta} is also fulfilled for the real-valued setting.
Then we can apply
Remark~\ref{lem:ConvLem1} and conclude that the minimizer~$y^\ast$
of~$\tilde{J}$ fulfills
$\lVert y^\ast - \tilde{f}\rVert_2 \leq \varepsilon < \frac{\pi}{16}$.
By~\eqref{eq:estCPPAiterates} we obtain
\begin{align} \label{eq:Nr1}
	R &= \sqrt{ \lVert y^\ast - \tilde{f}\rVert_2^2
			+ 2 \lVert \lambda\rVert_2^2 L^2 c (c+1)}
		+ 2 \lVert\lambda\rVert_\infty cL \\
	&\leq \sqrt{ \varepsilon^2 + 2 \lVert \lambda\rVert_2^2 L^2 c (c+1)}
		+ 2 \lVert\lambda\rVert_\infty cL  < \frac{\pi}{16}.
\end{align}
By Lemma~\ref{lem:ConvLem4} the iterates $y^{(k+\frac{l}{c})}$ of the
real-valued CPPA fulfill
\[
  \lVert y^{(k+\frac{l}{c})} - y^\ast\rVert_2  \le R < \frac{\pi}{16}.
\]
Hence~$\lVert y^{(k+\frac{l}{c})} - \tilde{f}\rVert_{\infty} < \frac{\pi}{8}$
which means that all iterates $y^{(k+\frac{l}{c})}$ stay within $\tilde{\mathcal S}(\tilde f,\frac{\pi}{8})$.

Next, we consider the sequence $\{ x^{(l+\frac{k}{c})} \}$ of the CPPA for the $\mathbb S^1$-valued data $f$.
We use induction to verify $x^{(k+\frac{l}{c})} = \exp_q(y^{(k+\frac{l}{c})})$.
By definition we have $x^{(0)} = f = \exp_q (\tilde f) = \exp_q(y^{(0)})$.
Assume that $x^{(k+\frac{l-1}{c})} = \exp_q(y^{(k+\frac{l-1}{c})})$.
By bijectivity of the lifting, cf. Lemma~\ref{lem:ConvLem2},
and since $y^{(k+\frac{l-1}{c})}\in\tilde{\mathcal S}(\tilde f,\delta)$,
we conclude $x^{(k+\frac{l-1}{c})}\in \mathcal S(f,\delta)$.
By Lemma~\ref{lem:ConvLem5} we obtain
\[
  \exp_q(y^{(k+\frac{l}{c})})
 = \exp_q\bigl(\prox_{\lambda_k\tilde J_{l}}(y^{(k+\frac{l-1}{c})})\bigr)
 = \prox_{\lambda_k J_{l}}( x^{(k+\frac{l-1}{c})}) = x^{(k+\frac{l}{c})}.
\]
By the same argument as above we have again $x^{(k+\frac{l}{c})} \in \mathcal S(f,\delta)$.

Finally, we know by Theorem~\ref{thm:bacak} that
\[
 x^{(k)} = \exp_q(y^{(k)}) \rightarrow \exp_q(y^\ast) \quad {\rm as} \quad k \rightarrow \infty
\]
and by Lemma~\ref{lem:ConvLem3} that $x^\ast := \exp_q(y^\ast)$ is a global minimizer of $J$.
This completes the proof.
\end{proof}
%
\section{Numerical results} \label{sec:numerics}
%
For the numerical computations of the following examples, the algorithms
presented in Section~\ref{sec:cpp} were implemented in \textsc{MatLab}. 
The computations were performed on a MacBook Pro with an Intel Core i5,
2.6\,Ghz and 8\,GB of RAM using \textsc{MatLab} 2013, Version 2013a (8.1.0.604)
on Mac OS 10.9.2.
%
\subsection{Signal denoising of synthetic data}
\begin{figure}\centering
	\begin{subfigure}{.48\textwidth}
	\centering
		\includegraphics{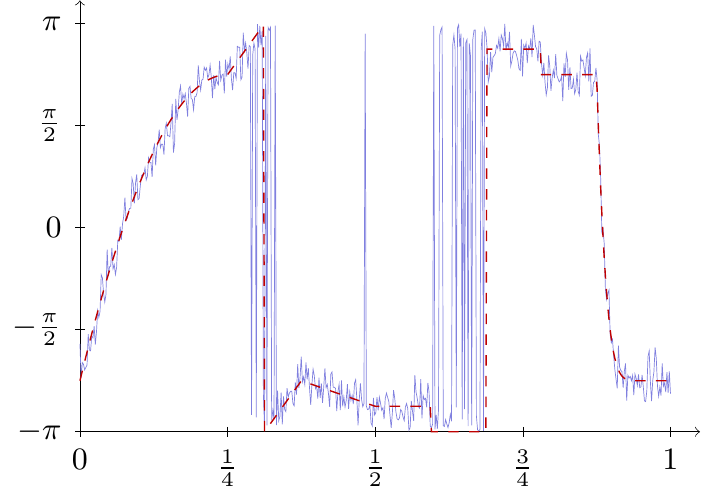}
		\caption{
		\(f_\text{o}\) and \(f_\text{n}\).
		}
		\label{subfig:noisy}
	\end{subfigure}
%
%
	\begin{subfigure}{.48\textwidth}\centering
			\includegraphics{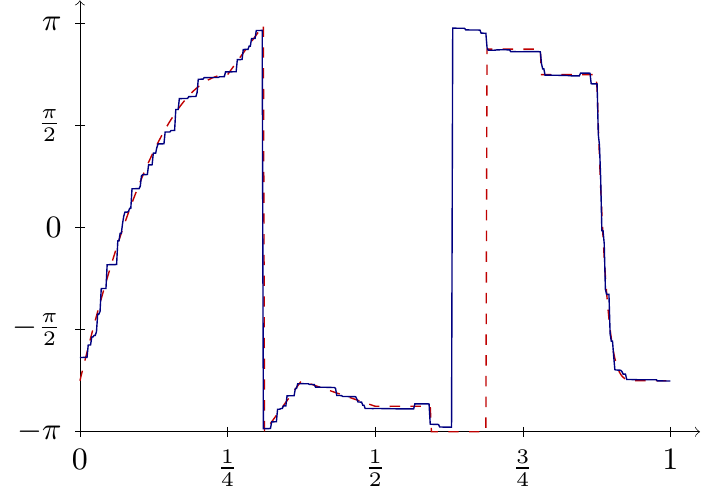}
		\caption{
\(\text{TV}_1\), \(e(f_\text{o},f_\text{r}) \approx 6.06\times10^{-3}\)\!.}
		\label{subfig:TV}
	\end{subfigure}
 	\begin{subfigure}{.48\textwidth}\centering
		\includegraphics{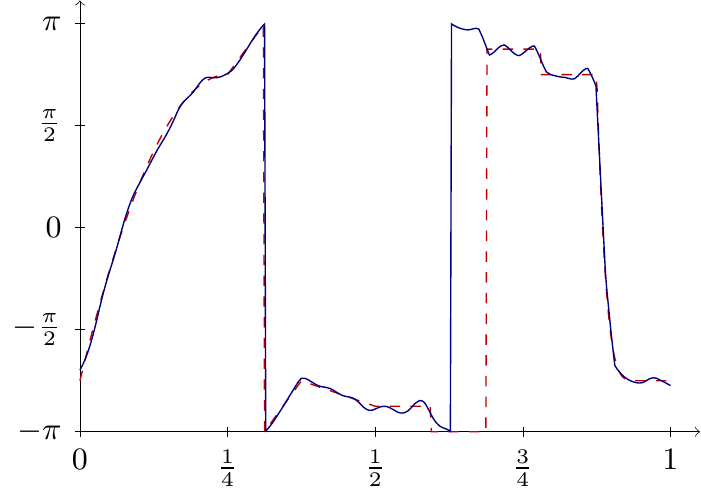}
			\caption{\(\text{TV}_2\),
			\(e(f_\text{o},f_\text{r}) \approx 4.34\times10^{-3}\)\!.}
		\label{subfig:TV2}
	\end{subfigure}
 	\begin{subfigure}{.48\textwidth}\centering
			\includegraphics{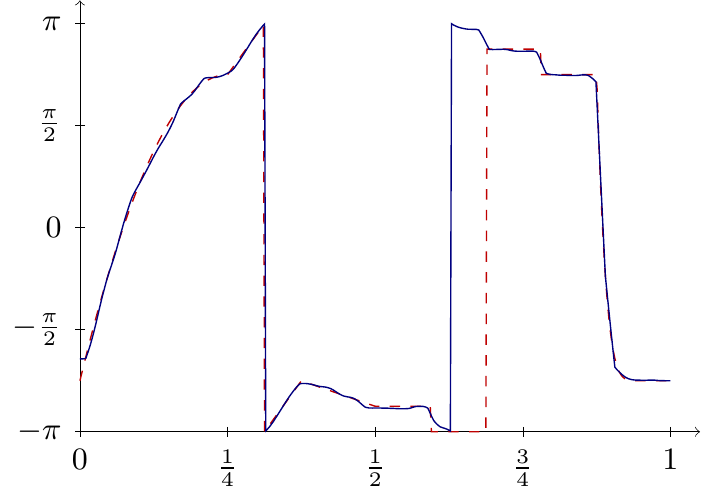}
			\caption{
\(\text{TV}_1\)\&\( \text{TV}_2 \),
			, \(e(f_\text{o},f_\text{r}) \approx 3.53\times10^{-3}\)\!.}
		\label{subfig:TV12}
	\end{subfigure}	
	\caption[]{Denoising of an one-dimensional signal by minimizing~\eqref{task_1} with CPPA. %
	\subref{subfig:noisy} 
	Original signal $f_\text{o}$ (dashed red) and disturbed signal by wrapped %
	Gaussian noise $f_\text{n}$ (solid black). %
	\subref{subfig:TV}---\subref{subfig:TV12} 
	Reconstructed signals $f_\text{r}$ using~\subref{subfig:TV}~
	only the $\text{TV}_1$ regularizer ($\alpha=\frac{3}{4}$), %
	\subref{subfig:TV2}~
	only the $\text{TV}_2$ regularizer ($\beta=\frac{3}{2}$), and %
	\subref{subfig:TV12}~
	both of them ($ \alpha=\frac{1}{2}$, $\beta=1$).
	While\,\subref{subfig:TV}~
	suffers from the staircasing effect,\,\subref{subfig:TV2}~
	shows weak results at constant areas. The combination of both regularizers %
	in\,\subref{subfig:TV12}~
	yields the best image.}\label{fig:Example1D}
\end{figure}
The first example of an artificial one-dimensional signal demonstrates the
effect of different models containing absolute cyclic
first order differences, second order differences or both combined.
The function~\(f: [0,1] \to [-\pi,\pi)\) given by
\[
	f(x) := \begin{cases}
		-24\pi x^2 + \frac{3}{4}\pi & \mbox{ for \(0\leq x\leq \frac{1}{4}\),}\\
		4\pi x - \frac{\pi}{4} &\mbox{ for \( \frac{1}{4}< x \leq \frac{3}{8}\),}\\
		\bigl(-\pi x - \frac{3}{8}\bigr)_{2\pi} &\mbox{ for \( \frac{3}{8}< x \leq \frac{1}{2}\),}\\
		\bigl(-\frac{j+7}{8}\pi\bigr)_{2\pi} & \mbox{ for \(\frac{3j+16}{32} < x \leq \frac{3j + 19}{32}\), \(j=0,1,2,3\),}\\
		\frac{3}{2}\pi\exp\bigl(-\frac{35}{7} -\frac{1}{1-x}\bigr) - \frac{3}{4}\pi & \mbox{ for \(\frac{7}{8} < x \leq 1\),}
	\end{cases}
\]
is sampled equidistantly to obtain the original signal
\(f_\text{o} = \Bigl(f\bigl(\frac{i-1}{N-1}\bigr)\Bigr)_{i=1}^N\)
at \(N=500\)
samples.
This function is distorted by wrapped Gaussian noise \(\eta\) of
standard deviation \(\sigma = \frac15\) to get
\( f_\text{n} := \bigl(f_\text{o} + (\eta)_{2\pi}\bigr)_{2\pi} = (f_\text{o} + \eta)_{2\pi}\),
see also Remark~\ref{lem:mod_pi}.
The functions $f_\text{o}$ and $f_\text{n}$ are depicted in Figure~\ref{subfig:noisy}.
Note the following effects due to the cyclic data representation on $[-\pi,\pi)$:
The linear increase on \(\bigl[\frac{1}{4},\frac{3}{8}\bigr]\) of $f$ is continuous and
the change from \(\pi\) to \(-\pi\) at $\frac{5}{16}$
is just due to the chosen representation system.
Similarly the two constant parts with the values \(-\pi\)
and~\(\frac{7}{8}\pi\) differ only by a jump
size of~\(-\frac{\pi}{8}\). For the noise around these two areas, we have the same
situation.

We apply Algorithm~\ref{alg:CPPA} with different model parameters $\alpha$ and $\beta$
to \(f_\text{n}\) which yields the restored signals \(f_\text{r}\).
The restoration error is measured by
the `cyclic' mean squared error (cMSE) with respect to the arc length distance
\[
	e(f_\text{o},f_\text{r}) := \frac{1}{N}\sum_{i=1}^M d(f_{\text{o},i},f_{\text{r},i})^2.
\]
We use $\lambda_0 = \pi$ and \(k=4000\) iterations as stopping criterion.
For any choice of parameters \(\alpha, \beta\) the computation time is about \( 6 \) seconds.

The result \(f_\text{r}\) in Figure~\ref{subfig:TV} is obtained using
only the \(\text{TV}_1\) regularization (\(\alpha=\frac{3}{4},\beta=0\)).
The restoration of constant areas is
favored by this regularization term, but linear, quadratic and
exponential parts suffer from the  well-known `staircasing' effect.
Utilizing only the \(\text{TV}_2\) regularization (\(\alpha=0,\beta = \frac23\)) ,
cf. Figure~\ref{subfig:TV2}, the restored function becomes worse in flat areas,
but shows a better quality in the linear parts.
By combining the regularization terms (\(\alpha=\frac{1}{2}\), \(\beta = 1\)) as
illustrated in Figure~\ref{subfig:TV12}
both the linear and the constant parts are reconstructed quite well and
the cMSE is smaller than for the other choices of parameters.
Note that $\alpha$ and $\beta$ were chosen in \(\frac{1}{4}\mathbb N\) with
respect to an optimal cMSE.
%
\subsection{Image denoising of InSAR data}
The complex-valued synthetic aperture radar (SAR) data is
obtained emitting specific radar signals at equidistant points and measuring
the amplitude and phase of their reflections by the earth's surface. The
amplitude provides information about the reflectivity of the surface. The phase
encodes both the change of the elevation of the surface's elements within the
measured area and their reflection properties and is therefore rather
arbitrary. When taking two SAR images of the target area at the  same time but
from different angles or locations. The phase difference of these images encodes
the elevation, but it is restricted to one wavelength and also includes noise. The
result is the so called interferometric synthetic aperture radar (InSAR) data
and consists of the `wrapped phase' or the `principal phase', a value in
\([-\pi,\pi)\) representing the surface elevation. For more details see,
e.g.,~\cite{BRF00,MF98}.

After a suitable preregistration the same approach can be applied to two images
from the same area taken at different points in time to measure surface
displacements, e.g., before and after an earthquake or the movement of glaciers.

The main challenge in order to unwrap the phase is the presence of noise.
Ideally, if the surface would be smooth enough and no noise would be present,
unwrapping is uniquely determined, i.e., differences between two pixels larger
than \(\pi\) are regarded as a wrapping result and hence become unwrapped.

There are several algorithms to unwrap, even combining the denoising and the
unwrapping, see for example~\cite{BKAE08,BV07}. For denoising, Deledalle et
al.~\cite{DDT11} use both SAR images and apply a non-local means algorithm
jointly to their reflection, the interferometric phase and
the coherence.

%
\begin{figure}[tbp]\centering
		\captionsetup[sub]{skip=0pt}
		\setlength{\abovecaptionskip}{-3pt}
		\vspace{-\baselineskip}
		\begin{subfigure}{.48\textwidth}\centering
			\includegraphics{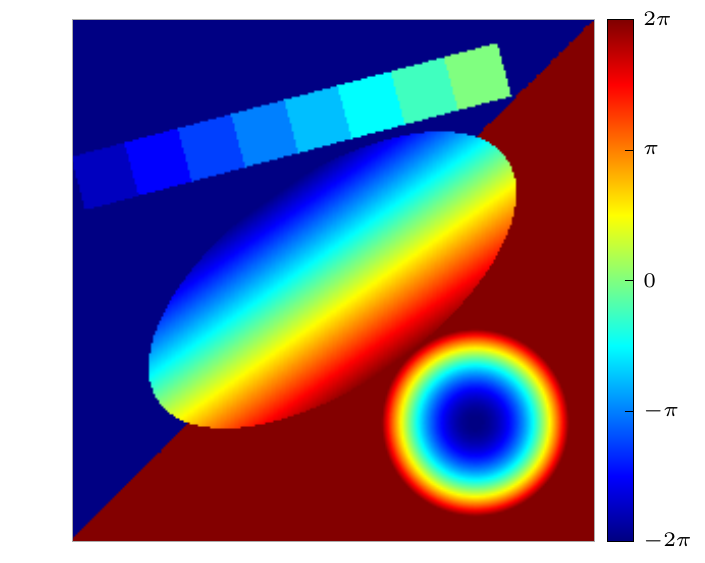}
			\vspace{-.5\baselineskip}
			\caption{Original data.}\label{subfig:2Dorig}
		\end{subfigure}
		\begin{subfigure}{.48\textwidth}\centering
			\includegraphics{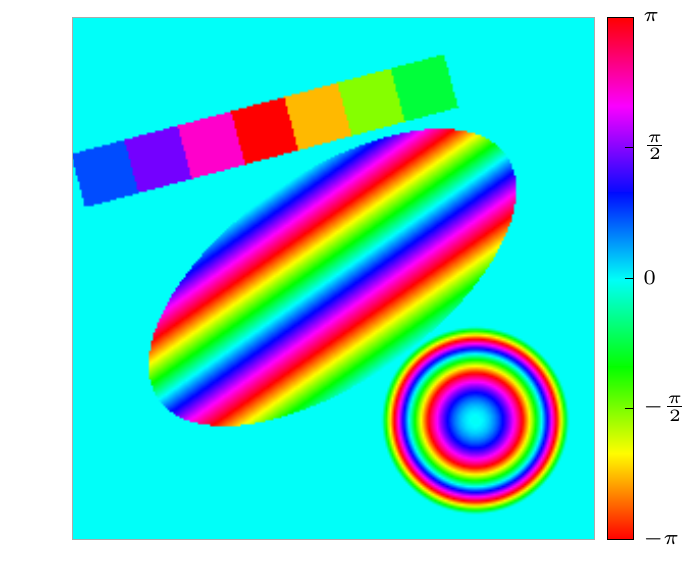}
			\vspace{-.5\baselineskip}
			\caption{Wrapped data $f_{\text{o}} $.}\label{subfig:2Dwrapped}
		\end{subfigure}
		\begin{subfigure}{.48\textwidth}\centering
			\includegraphics{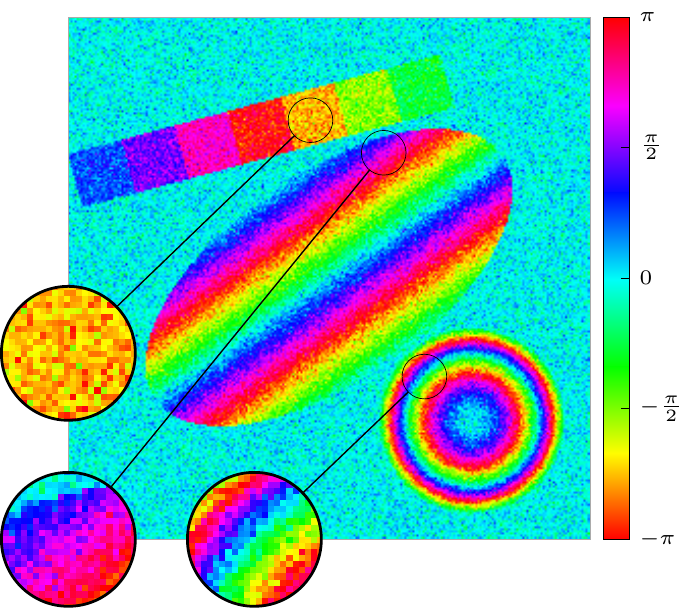}
			\caption{Noisy data $f_{\text n}$.
			}\label{subfig:2Dnoisy}
		\end{subfigure}
		\begin{subfigure}{.48\textwidth}\centering
			\includegraphics{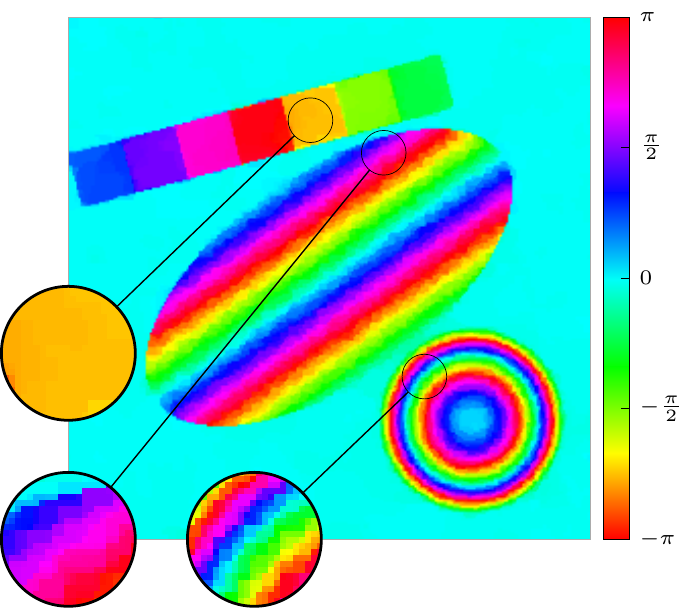}
			\caption{
			$\text{TV}_1$,
			$e(f_{\text{o}},f_\text{r}) = 
			7.09\,\times10^{-3}$.}\label{subfig:2DTV1}
		\end{subfigure}
		\begin{subfigure}{.48\textwidth}\centering
			\includegraphics{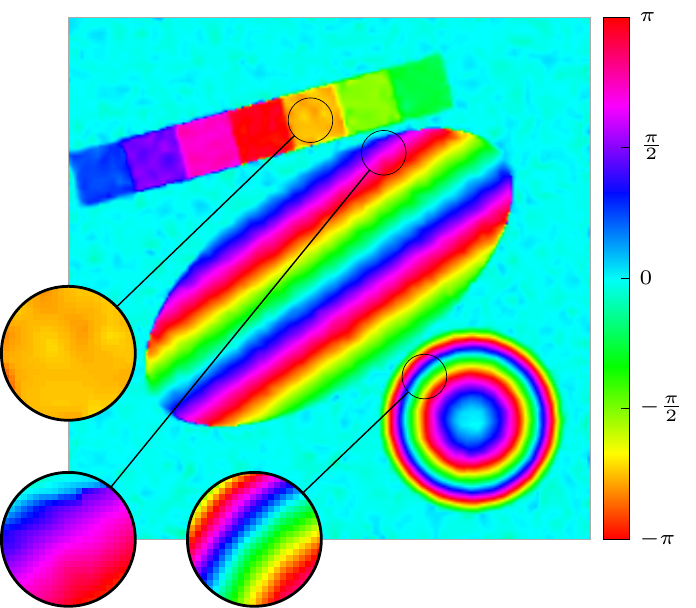}
			\caption{$\text{TV}^{\mathrm{hv}}_2$,
			$e(f_{\text{o}},f_\text{r}) = 
			6.70\,\times\,10^{-3}$.}\label{subfig:2DTV2}
		\end{subfigure}
		\begin{subfigure}{.48\textwidth}\centering
				\includegraphics{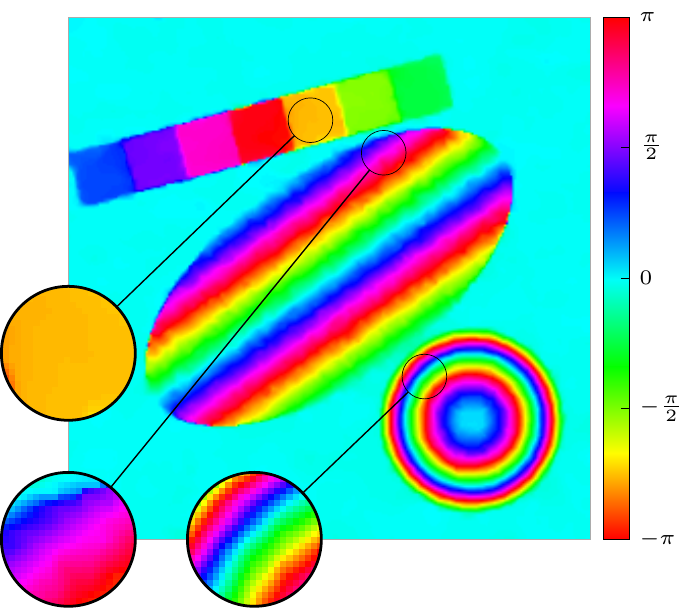}
			\caption{$\text{TV}_1$ \& $\text{TV}^{\mathrm{hv}}_2$,
			$e(f_{\text{o}},f_\text{r}) = 
			5.37\,\times\,10^{-3}$.}\label{subfig:2DTV12}
		\end{subfigure}
		\vspace{\baselineskip}
		\caption[]{Denoising of two-dimensional artificial data by
		minimizing~\eqref{task_2} with CPPA.
		\subref{subfig:2Dorig}~
		Artificial surface,\,\subref{subfig:2Dwrapped}~
		its wrapped variant, and\,\subref{subfig:2Dnoisy}~
		wrapped image corrupted by wrapped Gaussian noise.
		\subref{subfig:2DTV1}--\subref{subfig:2DTV12}~
		Reconstructed images $f_\text{r}$ using\,\subref{subfig:2DTV1}~
		only the $\operatorname{TV}_1$ regularizer ($\alpha = (\frac{3}{8}, \frac{1}{4})$),
		\subref{subfig:2DTV2}~
		only  the $\operatorname{TV}^{\text{hv}}_2\,\&\;\operatorname{TV}^{\text{d}}_2$ regularizer ($\beta = (\frac{1}{8},\frac{1}{8})$, $\gamma=\frac{1}{8}$),
		and\,\subref{subfig:2DTV12}~
		both of them ($\alpha=(\frac{1}{4},\frac{1}{8})$, $\beta= (\frac{1}{8},\frac{1}{8})$, $\gamma=0$).
}
\vspace{-16pt}
\end{figure}
\paragraph{Application to synthetic data.}
In order to get a better understanding in the two-dimen\-sio\-nal case, let us
first take a look at a synthetic surface given on~\([0,1]^2\) with the profile
shown in Figure~\ref{subfig:2Dorig}. This surface consists of two plates of
height~\(\pm2\pi\) divided at the diagonal, a set of stairs in the upper left
corner in direction~\(\frac{\pi}{3}\), a linear increasing area
connecting both plateaus having the shape of an ellipse with major
axis at the angle~\(\frac{\pi}{6}\), and a half ellipsoid forming a dent in the
lower right of the image with circular diameter of size~\(\frac{9}{25}\) and
depth~\(4\pi\).
The initial data is given by sampling the described surface at~\(M=N=256\)
sampling points.
The usual InSAR measurement would ideally result in data as given in
Figure~\ref{subfig:2Dwrapped}, i.e., the data is wrapped with respect to
\(2\pi\).
In the figure the resulting ideal phase is represented using the hue component
of the HSV color space.
Again, the data is perturbed by wrapped Gaussian noise,
standard deviation~\(\sigma = 0.3\), see Figure~\ref{subfig:2Dnoisy}.

For an application of Algorithm~\ref{alg:CPPA} to the minimization problem~\eqref{task_2},
we have to fix five parameters \(\alpha_1,\alpha_2,\beta_1,\beta_2,\gamma\)
which were chosen on~\( \frac{1}{8}\mathbb N \) such that they minimize the cMSE.
Using only the cyclic first order differences with \(\alpha=\frac{1}{8}(3,2)\),
see Figure~\ref{subfig:2DTV1}, the reconstructed image $f_{\text{r}}$ reproduces
the piecewise constant parts of the stairs in the upper left part and the
background, but introduces a staircasing in both linear increasing areas inside
the ellipse and in the half ellipsoid.
This is highlighted in the three magnifications in
Figure~\ref{subfig:2DTV1}.
Applying only cyclic second order differences with
\(\beta_1=\beta_2=\gamma=\frac{1}{8}\) manages to reconstruct the
linear increasing part and the circular structure of the ellipsoid, but
compared to the first case it even increases the cMSE due to the approximation
of the stairs and the background, see especially the magnification of the stairs
in Figure~\ref{subfig:2DTV2}.
Combining first and second order cyclic differences by
setting~\(\alpha_1=\alpha_2=\frac{1}{8}(2,1)\)
and~\(\beta_1=\beta_2=\frac{1}{8}\), \(\gamma=0\), these disadvantages can be
reduced, cf. Figure~\ref{subfig:2DTV12}. Note especially the three
magnified regions and the cMSE.
%
\begin{figure}[tbp]\centering
	\begin{subfigure}{.48\textwidth}\centering
			\includegraphics{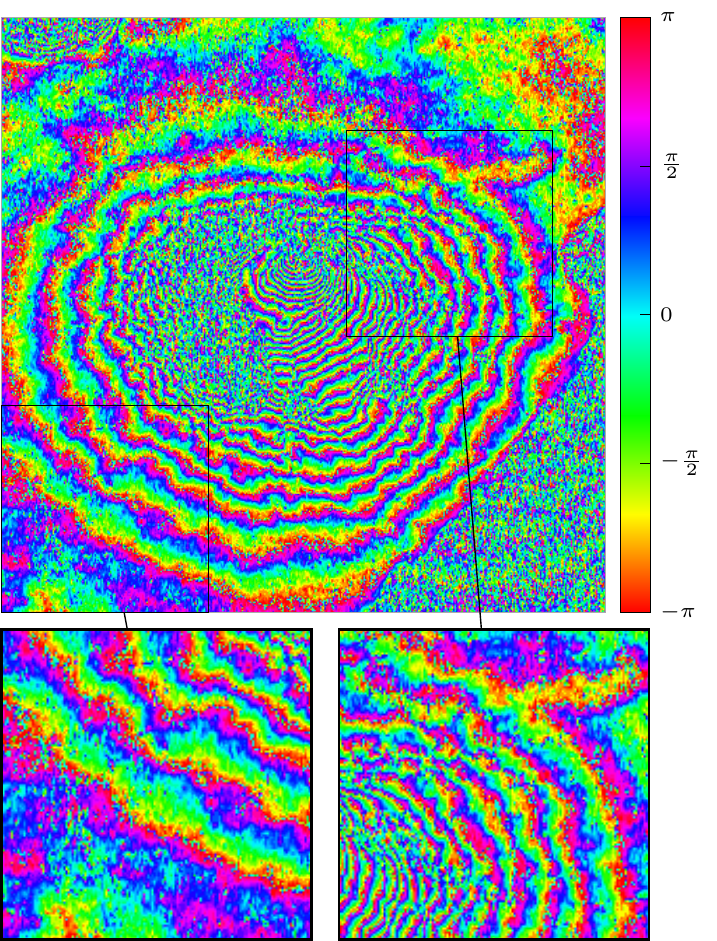}
		\caption{Complete Vesuvuis data.}\label{subfig:vesuv-orig}
	\end{subfigure}
	\begin{subfigure}{.48\textwidth}\centering
			\includegraphics{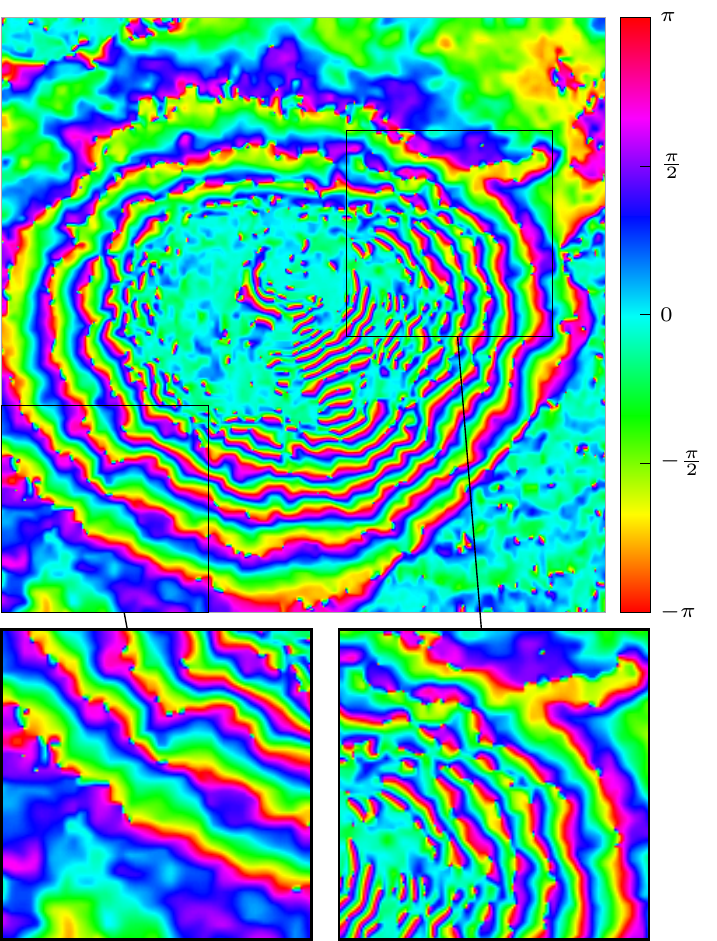}
		\caption{Denoised Vesuvius data.}\label{subfig:vesuv-denoised}
	\end{subfigure}
	\caption[]{\subref{subfig:vesuv-orig}~
Noisy InSAR data set of the Vesuvius taken by the ERS-1 satellite~\cite{RPG97}.\,\subref{subfig:vesuv-denoised}~
Denoised image by minimizing~\eqref{task_2} with CPPA ($\alpha= \bigl(\frac{1}{4},\frac{1}{4}\bigr)$, $\beta=\bigl(\frac{3}{4},\frac{3}{4}\bigr)$ and $\gamma=\frac{3}{4}$).}
\label{fig:Vesuvius}
\end{figure}
\paragraph{Application to real-world data.}
Next we examine a real-world example. The data from
\cite{RPG97} is a set of InSAR data recorded in 1991 by the ERS-1 satellite
capturing topographical information from the Mount Vesuvius. The data is
available online\footnote{at %
\url{https://earth.esa.int/workshops/ers97/program-details/speeches/rocca-et-al/}%
}
and a part of it was also used as an example in~\cite{WDS2013} for TV based
denoising of manifold-valued data.
In Figure~\ref{fig:Vesuvius} the phase is
represented by the hue component of the HSV color space.
We apply Algorithm~\ref{alg:CPPA} to the image of size~$426 \times 432$, cf.
Figure~\ref{subfig:vesuv-orig}, with~
\(\alpha_1=\alpha_2=\frac{1}{4}\) and~\(\beta_1=\beta_2=\gamma=\frac{3}{4}\).
This reduces the noise while keeping all significant plateaus, ascents and
descents, cf. Figure~\ref{subfig:vesuv-denoised}. The left zoom illustrates how
the plateau in the bottom left of the data is smoothened
but kept in its main elevation shown in blue. In the zoom on the right all
major parts except the noise are kept. We notice just a little smoothening due
to the linearization introduced by~\(\operatorname{TV}_2\). In the  bottom left
of this detail some of the fringes are eliminated, and a small plateau is build
instead, shown in cyan. The computation time for the
whole image  using \(k=600\) iterations as
stopping criterion was 86.6 sec and 11.1 sec for each of the details of size
\(150\times 150\).
\section{Conclusions} \label{sec:conclusions}
%
In this paper we considered functionals having regularizers with
second order absolute cyclic differences for $\mathbb S^1$-valued data.
Their definition required a proper notion of higher order differences of cyclic
data generalizing the corresponding concept in Euclidian spaces.
We derived a CPPA for the minimization of
our functionals and gave the explicit expressions for the appearing proximal
mappings.
We proved convergence of the CPPA under certain conditions.
To the best of our knowledge this is the first algorithm dealing with higher
order TV-type minimization for $\mathbb S^1$-valued data.
We demonstrated the denoising capabilities of our model on synthetic as well as
on real-world data.

Future work includes the application of our higher order methods
for cyclic data to other imaging tasks such as
segmentation, inpainting or deblurring.
For deblurring, the usually underlying linear convolution kernel
has to be replaced by a nonlinear construction based on intrinsic (also called Karcher) means.
This leads to the task of solving the new associated inverse problem.

Further, we intend to investigate other couplings of first and second order
derivatives similar to infimal convolutions or GTV for Euclidean data.
Finally, we  want to set up  higher order TV-like methods for more general
manifolds, e.g. higher dimensional spheres.
Here, we do not believe that it is possible to derive explicit expressions for
the involved proximal mappings -- at least not for Riemannian manifolds of
nonzero sectional curvature. Instead, we plan to resort to iterative techniques.
%
%
\paragraph{Acknowledgement.} We acknowledge the financial support by DFG Grant STE571/11-1.

\end{document}